\documentclass[11pt]{amsart}

\usepackage{amsmath,
amssymb,
amscd,
amsfonts,
graphicx,
latexsym,
epsfig,
psfrag,
url,
color,
xr,
mathtools,
enumitem,
transparent,
hyperref,
float,
mathrsfs,
wrapfig,
rotating,
array,
import,
tikz-cd,
caption}
\usepackage{fullpage}
\usepackage[all]{xy}
\usepackage[foot]{amsaddr}
\usepackage[super]{nth}

\newtheorem{theorem}{Theorem}[section]
\newtheorem{corollary}[theorem]{Corollary}

\newtheorem{lemma}[theorem]{Lemma}

\theoremstyle{definition}
\newtheorem{definition}[theorem]{Definition}
\theoremstyle{definition}

\newtheorem{remark}[theorem]{Remark}

\newtheorem{example}[theorem]{Example}
\theoremstyle{plain}
\newcommand{\thistheoremname}{}
\newtheorem{genericthm}[theorem]{\thistheoremname}

\newtheorem*{genericthm*}{\thistheoremname}
\newenvironment{namedthm*}[1]
  {\renewcommand{\thistheoremname}{#1}%
   \begin{genericthm*}}
  {\end{genericthm*}}
  
\newcommand\cG{\mathcal{G}}

\newcommand\cM{\mathcal{M}}

\newcommand\cS{\mathcal{S}}

\newcommand{\bA}{\mathbb{A}}
\newcommand{\bC}{\mathbb{C}}

\newcommand{\bP}{\mathbb{P}}

\newcommand{\bR}{\mathbb{R}}

\newcommand{\bZ}{\mathbb{Z}}

\newcommand{\bCP}{\mathbb{CP}}

\newcommand\ba{\mathbf{a}}
\newcommand\bb{\mathbf{b}}

\newcommand\bm{\mathbf{m}}
\newcommand\bn{\mathbf{n}}

\newcommand\boldr{\mathbf{r}}
\newcommand\bs{\mathbf{s}}
\newcommand\bv{\mathbf{v}}
\newcommand\bw{\mathbf{w}}
\newcommand\bx{\mathbf{x}}

\newcommand\bz{\mathbf{z}}

\newcommand\bzero{\mathbf{0}}

\newcommand{\btB}{{\mathbf{2B}}}
\newcommand{\sB}{\mathscr{B}}
\newcommand{\stB}{2\mathscr{B}}
\newcommand{\sP}{\mathscr{P}}

\newcommand{\on}{\operatorname}

\newcommand\id{\on{id}}

\newcommand{\comp}{C^2}

\renewcommand{\comp}{{\on{comp}}}
\newcommand{\seam}{{\on{seam}}}
\newcommand{\mk}{{\on{mark}}}
\newcommand{\incom}{\on{in}}

\newcommand{\inte}{{\on{int}}}

\renewcommand{\root}{{\on{root}}}

\newcommand{\rk}{{\on{rk}\:}}

\newcommand{\tree}{{\on{tree}}}
\newcommand{\br}{{\on{br}}}

\renewcommand{\top}{{\on{top}}}
\newcommand{\Bl}{{\on{Bl}}}
\newcommand{\SC}{\mathcal{SC}}
\newcommand{\SPT}{\mathcal{SPT}}
\newcommand{\stab}{{\on{stab}}}

\newcommand{\an}{{\on{an}}}
\newcommand{\Gr}{{\on{Gr}}}
\newcommand{\Spec}{\on{Spec}}
\newcommand{\amb}{{\on{amb}}}
\newcommand{\sat}{{\on{sat}}}

\makeatletter
\newcommand*\bigcdot{{\mathpalette\bigcdot@{.5}}}
\newcommand*\bigcdot@[2]{{\mathbin{\vcenter{\hbox{\scalebox{#2}{$\m@th#1\bullet$}}}}}}
\makeatother

\newcommand\qu{/\kern-.7ex/} 
\newcommand\lqu{\backslash \kern-.7ex \backslash}

\newcommand{\ol}{\overline}
\newcommand{\ul}{\underline}

\newcommand{\sr}{\stackrel}

\newcommand{\wt}{\widetilde}

\newcommand{\eps}{\epsilon}

\def\hra{\hookrightarrow}
\def\lra{\longrightarrow}

\newcounter{qcounter}

\newcommand\quotient[2]{
        \mathchoice
            {
                \text{\raise1ex\hbox{$#1$}\Big/\lower1ex\hbox{$#2$}}%
            }
            {
                #1\,/\,#2
            }
            {
                #1\,/\,#2
            }
            {
                #1\,/\,#2
            }
    }

\newcommand\quoti[2]{
                \text{\raise1ex\hbox{$#1$}/\lower1ex\hbox{$\scriptstyle#2$}}
  }

\newcommand\quot[2]{
                \text{\raise1ex\hbox{$#1\!\!$}/\lower1ex\hbox{$\!\scriptstyle#2$}}
  }

\newcommand\quo[2]{
                \text{\raise.8ex\hbox{$\scriptstyle#1\!$}/\lower.8ex\hbox{$\!\scriptstyle#2$}}
  }

\newcommand\qq[2]{
                \text{\raise.8ex\hbox{$#1\!$}/\lower.8ex\hbox{$#2$}}
}

\newcommand{\fm}{Fulton--MacPherson }

\begin{document}

\title{A compactification of the moduli space of marked vertical lines in $\bC^2$}
\author{Nathaniel Bottman}
\address{Department of Mathematics, University of Southern California, 3620 S Vermont Ave, Los Angeles, CA 90089}
\email{\href{mailto:bottman@usc.edu}{bottman@usc.edu}}
\author{Alexei Oblomkov}
\address{Department of Mathematics and Statistics, University of Massachusetts at
Amherst, Lederle Graduate Research Tower, 710 N Pleasant Street, Amherst, MA 01003 USA}
\email{\href{mailto:oblomkov@math.umass.edu}{oblomkov@math.umass.edu}}
\date{\today} 

\maketitle

\begin{abstract}
For $r \geq 1$ and $\bn \in \bZ_{\geq0}^r\setminus\{\bzero\}$, we construct a proper complex variety $\ol{2M}_\bn$.
$\ol{2M}_\bn$ is locally toric, and it is equipped with a forgetful map $\ol{2M}_\bn \to \ol M_{0,r+1}$.
This space is a compactification of $2M_\bn$, the configuration space of marked vertical lines in $\bC^2$ up to translations and dilations.
In the appendices, we give several examples and show how the stratification of $\ol{2M}_\bn$ can be used to recursively compute its virtual Poincar\'{e} polynomial.
\end{abstract}

\section{Introduction}


For $r\geq1$ and $\bn \in \bZ_{\geq0}^r\setminus\{\bzero\}$, the first author constructed in \cite{b:2ass} a poset $W_\bn=W_\bn^\bR$ called a \emph{2-associahedron}.
In \cite{b:realization}, he constructed a compact, metrizable topological space $\ol{2M}_\bn=\ol{2M}_\bn^\bR$, which is stratified by $W_\bn$.
The realizations $\ol{2M}_\bn^\bR$ will play the role of domain moduli spaces in the first author's chain-level functoriality structure for the Fukaya category.

In this paper, we construct algebraic varieties $\ol{2M}_\bn^\bC$ that complete the following analogy:
\begin{align}
(r-2)\text{-dim.\ associahedron}
\::\:
\ol M_{0,r+1}
\:::\:
\ol{2M}_\bn^\bR
\::\:
\ol{2M}_\bn^\bC.
\end{align}
That is, one construction of the associahedron resp.\ $\ol M_{0,r+1}$ is as the compactified moduli space of $r$ unmarked points in $\bR$ resp.\ marked points in $\bC$.
$\ol{2M}_\bn^\bR$ is the compactified moduli space of $r$ vertical lines in $\bR^2$ with $n_i$ marked points on the $i$-th line from the left, so in an analogous way we define $\ol{2M}_\bn^\bC$ to be the compactified moduli space of $r$ labeled vertical complex lines in $\bC^2$ with $n_i$ labeled marked points on the $i$-th line.
Our main result is to construct $\ol{2M}_\bn^\bC$ as a variety with mild singularities:
\begin{theorem}
\label{thm:2Mn-bar}
Fix $r \geq 1$ and $\bn \in \bZ_{\geq0}^r\setminus\{\bzero\}$.
Then $\ol{2M}_\bn^\bC$, equipped with the atlas we define in \S\ref{ss:2Mn-bar_charts}, is a proper complex variety with toric singularities.
There is a forgetful morphism $\pi\colon \ol{2M}_\bn^\bC \to \ol M_{0,r+1}$, which on the open locus sends a configuration of lines and points to the positions of the lines, thought of as a configuration of points in $\bC$.
\end{theorem}

\noindent
For instance, in Example~\ref{ex:2Mn_model}, we show that one of the local models in our atlas is the following quadric cone, which is indeed a toric affine complete intersection:
\begin{align}
\left\{(a,b,c,d,e,f) \in \bC^6
\:\left|\:
{
{c=d,e=f}
\atop
{ac=ad=be=bf}
}
\right.\right\},
\end{align}

We now describe the plan for our paper.
\begin{enumerate}
\item[\bf\S\ref{s:Mr-bar}:]
We begin our paper by constructing a smooth proper complex variety $\ol M_r^\bC$, which is isomorphic to $\ol M_{0,r+1}$ (we do not construct this isomorphism here).
While $\ol M_{0,r+1}$ is obviously not a new space, we give a construction in terms of an atlas that, to our knowledge, has not appeared in the literature.
This construction serves as a warm-up for the atlas we construct in \S\ref{s:Wn-bar}, which is key to our proof of Thm.~\ref{thm:2Mn-bar}.

\medskip

\item[\bf\S\ref{s:Wn-bar}:]
In this, the central section of our paper, we define $\ol{2M}_\bn^\bC$ and prove Thm.~\ref{thm:2Mn-bar}.
We provide numerous illustrations of the constructions that we introduce.

\medskip

\item[\bf\S\ref{sec:wond-comp-}:]
We describe some future directions, which we will pursue in our follow-up paper \cite{bo:fm}.
Specifically, we plan to cast $\ol{2M}_\bn$ as part of a general construction of ``Fulton--MacPherson compactification for pairs $X \to Y$.''

\medskip

\item[\bf\S\ref{s:examples}:]
We work out all 1- and 2-dimensional instances of $\ol{2M}_\bn^\bC$.

\medskip

\item[\bf\S\ref{s:vpp}:]
There is a stratification on $\ol{2M}_\bn^\bC$ according to the combinatorial type of the corresponding trees of surfaces.
We use this stratification to recursively compute the virtual Poincar\'{e} polynomial of $\ol{2M}_\bn^\bC$, which we have implemented in \textsc{Python} at \cite{bo:program}.
\end{enumerate}


%

\subsection{Acknowledgments}

The first author was supported by an NSF Mathematical Sciences Postdoctoral Research Fellowship and by an NSF Standard Grant (DMS-1906220).
The second  was partially supported by an NSF CAREER Grant
(DMS-1352398), an NSF FRG Grant (DMS-1760373) and the Simons Foundation.
Both authors thank the Mathematical Sciences Research Institute, where this project began, and the first author thanks the Institute of Advanced Study for its hospitality.
Conversations with Satyan Devadoss and Helge Ruddat motivated the first author to begin thinking about a complex analogue of the 2-associahedra, and Paul Seidel also provided encouragement.
Ruddat, drawing on ideas from \cite[\S4]{helge}, made the first checks in GAP that the local models for $\ol{2M}_\bn^\bC$ are reduced and normal.
Conversations with Dominic Joyce helped the first author understand the notion of smooth manifolds with generalized corners.
Felix Janda helped with the proof of Lemma~\ref{lem:ineqs_real_to_int}. The second author would like to thank Paul Hacking and Jenia Tevelev for useful discussions.

\section{Construction of $\ol M_r^\bC$ via explicit charts}
\label{s:Mr-bar}

In this section, we will provide a construction of the compactified moduli space $\ol M_{0,r+1}$ of $(r+1)$-pointed genus-0 curves.
For consistency with \S\ref{s:Wn-bar}, we use the alternate notation $\ol M_r^\bC$.
While $\ol M_{0,r+1}$ is obviously a well-known and -studied space, to our knowledge, our construction has not appeared in the literature.
This section will also serve as a warm-up for our construction of $\ol{2M}_\bn^\bC$ in \S\ref{s:Wn-bar}.
The current section draws on the common approach in the symplectic geometry literature of constructing moduli spaces via an atlas, in which each chart is defined by considering ``gluing parameters'' which govern whether and how to smooth the nodes.

The following is the main result of this section:

\begin{theorem}
\label{thm:Mr_properties}
Fix $r\geq2$.
Then $\ol M_r^\bC$, equipped with the atlas we define in \S\ref{ss:Mr_charts}, is a smooth proper complex variety.
\end{theorem}

\begin{proof}
We show in Lemma~\ref{lem:Mr_variety} that the atlas we define in \S\ref{ss:Mr_charts} endows $\ol M_r^\bC$ with the structure of a prevariety.
Next, we show in Lemma~\ref{lem:Mr_sep_proper} that $\ol M_r^\bC$ is separated (hence a variety) and proper.
Finally, the domains $X_T$ of the charts are Zariski-open subsets of affine space, so $\ol M_r^\bC$ is smooth.
\end{proof}

\subsection{Definition of $K_r^\bC$}
\label{ss:Kr}

Before we can define $\ol M_r^\bC$, we must define the poset $K_r^\bC$, which will index the strata of $\ol M_r^\bC$.
$K_r^\bC$ consists of all possible combinatorial types of the elements of $\ol M_r^\bC$, and we present two equivalent definitions, one in terms of stable rooted trees with labeled leaves, and one in terms of bracketings of $r$ letters.

\begin{definition}
\label{def:Krtree_set}
We say that a vertex $\alpha$ of a rooted tree $T$ is \emph{interior} if the set $\incom(\alpha)$ of its incoming neighbors is nonempty, and we denote the set of interior vertices of $T$ by $T_\inte$ or $V_\inte(T)$.
A rooted tree $T$ is \emph{stable} if, when it is oriented toward its root, every interior vertex has at least 2 incoming edges.
We define $K_r^{\bC,\tree}$ to be the set of all isomorphism classes of stable rooted ribbon trees with $r$ labeled leaves, and we equip $K_r^{\bC,\tree}$ with the structure of a poset by declaring $T' < T$ if $T$ can be obtained by contracting a collection of interior edges of $T'$.
We denote the unique maximal element of $K_r^{\bC,\tree}$ by $T_r^\top$.

We denote the $i$-th leaf of $T \in K^\bC_r$ by $\lambda_i^T$.
For any $\rho, \sigma \in T$, $T_{\rho\sigma}$\label{p:Trhosigma} denotes those vertices $\tau$ such that the path $[\rho,\sigma]$\label{p:path} from $\rho$ to $\sigma$ passes through $\tau$.
We denote $T_\rho \coloneqq T_{\rho_\root\rho}$.
\null\hfill$\triangle$
\end{definition}

\begin{definition}
\label{def:Krbr}
A \emph{1-bracket of $r$} is a nonempty subset $B \subset \{1,\ldots,r\}$. \label{p:B}
A \emph{1-bracketing of $r$} is a collection $\sB$\label{p:sB} of 1-brackets of $r$ satisfying these properties:
\begin{itemize}
\item[] {\sc (Bracketing)} If $B, B' \in \sB$ have $B \cap B' \neq \emptyset$, then either $B \subset B'$ or $B' \subset B$.

\item[] {\sc (Root and leaves)} $\sB$ contains $\{1,\ldots,r\}$ and $\{i\}$ for every $i$.
\end{itemize}
We denote the set of all 1-bracketings of $r$ by $K_r^{\bC,\br}$,\label{p:Krbr} and define a partial order by defining $\sB' < \sB$ if $\sB$ is a proper subcollection of $\sB'$.
\null\hfill$\triangle$
\end{definition}

\noindent
An argument exactly analogous to the proof of \cite[Prop.\ 2.13]{b:2ass} shows that $K_r^{\bC,\tree}$ and $K_r^{\bC,\br}$ are isomorphic posets, and we denote $K_r^\bC \coloneqq K_r^{\bC,\tree} = K_r^{\bC,\br}$.

Next, we define, for any stable rooted tree $T \in K_r^\bC$, the poset $p_T$:
\begin{gather}
p_T
\coloneqq
\{0,1\}^{\#V_\inte(T) - 1},
\end{gather}
equipped with the usual poset structure.
(This can be thought of as a ``local model'' for $K_r^\bC$, when $K_r^\bC$ is viewed as the poset of strata of $\ol M_r^\bC$.)
We denote a typical element by $\boldr = (r_\rho)_{\rho \in V_\inte(T)\setminus\{\rho_\root\}}$.
(Note that a component $r_\rho$ being 0 resp.\ 1 indicates that we should not glue resp.\ should glue at that vertex.)

\begin{definition}
\label{def:gT}
Define a map $g_T\colon p_T \to K_r^\bC$ by defining $g_T(\boldr)$ to be the result of contracting each edge of $T$ whose incoming vertex $\alpha$ has $r_\alpha = 1$.
$\null\hfill\triangle$
\end{definition}

\noindent
Note that $V_\inte(g_T(\boldr))$ can be identified with $g_T^{-1}\{0\} \cup \{\rho_\root\}$.
Moreover, for any $\rho \in g_T^{-1}\{0\} \cup \{\rho_\root\}$, the following equality holds:
\begin{align}
\label{eq:in_of_gT}
\incom_{g_T(\boldr)}(\rho)
=
\left\{\sigma \in g_T^{-1}\{0\} \cup \{\rho_\root\}
\:\left|\:
[\rho,\sigma]_T \cap g_T^{-1}\{0\} \subset \{\rho,\sigma\}\right.\right\}.
\end{align}
(Here and elsewhere, we use the notation $[\rho,\sigma]$ to denote the path between $\rho$ and $\sigma$.)

Finally, we show that $g_T$ is a poset injection and characterize its image.

\begin{lemma}
\label{lem:gT_is_inclusion}
For any $T \in K_r^{\bC,\tree}$, $g_T$ is an inclusion of posets, with image $g(P_T) = [T,T_r^\top]$.
\end{lemma}

\begin{proof}
Immediate from the definition of the partial order on $K_r^{\bC,\tree}$.
\end{proof}

\subsection{Construction of $\ol M_r^\bC$ as a set, and of an atlas}
\label{ss:Mr_charts}

In this subsection, we construct $\ol M_r^\bC$ as a set and equip it with an atlas.
We will use this atlas in \S\ref{ss:Mr-bar-charts} to equip $\ol M_r^\bC$ with the structure of an algebraic variety.

\begin{definition}
\label{def:SDT}
A \emph{stable curve with $r \geq 2$ input marked points} is a pair
\begin{align}
\Bigl(T, (x_{\rho\sigma})_{
{\rho \in V_\inte(T),}
\atop
{\sigma \in \incom(\rho)}
}\Bigr),
\end{align}
where:
\begin{itemize}
\item $T$ is a stable rooted tree with $r$ leaves.

\item For $\rho \in V_\inte(T)$, $\bx_\rho \coloneqq (x_{\rho\sigma})_\sigma$ is an element of $\bC^{\#\incom(\rho)} \setminus \Delta$, where $\Delta$ is the big diagonal.
\end{itemize}
We say that two stable curves $\bigl(T,(x_{\rho\sigma})\bigr)$, $\bigl(T',(x'_{\rho\sigma})\bigr)$ are \emph{isomorphic} if there is an isomorphism of rooted trees $f\colon T \to T'$ and a function $V_\inte(T) \to G_1\colon \rho \mapsto \phi_\rho$ (where $G_1$ is the reparametrization group $\bC \rtimes (\bC\setminus\{0\})$ acting on $\bC$ by translations and dilations)
such that:
\begin{gather}
x'_{f(\rho)f(\sigma)} = \phi_\rho(x_{\rho\sigma}) \:\:\forall\:\: \rho \in V_\inte(T), \: \sigma \in \incom(\rho).
\end{gather}
We extend the notation $x_{\rho\sigma}$ to allow any distinct $\rho \in V_\inte(T)$ and $\sigma \in V(T)$, like so:
\begin{itemize}
\item
Suppose $\sigma$ lies in $T_\rho$.
Set $\tau$ to be the first vertex after $\rho$ through which the path from $\rho$ to $\sigma$ passes.
Define $x_{\rho\sigma} \coloneqq x_{\rho\tau}$.

\medskip

\item
If $\sigma$ does not lie in $T_\rho$, then set $x_{\rho\sigma} \coloneqq \infty$.
\end{itemize}

We denote by $\SC_r$ the collection of stable curves with $r$ input marked points, and we define the \emph{moduli space of stable curves with $r$ input marked points $\ol M_r^\bC$} to be the set of isomorphism classes of stable curves of this type.
For any stable rooted tree $T$ with $r$ leaves, define the corresponding \emph{strata} $\SC_{r,T} \subset \SC_r$, $\ol\cM_{r,T} \subset \ol M_r$ to be the set of all stable curves (resp.\ isomorphism classes thereof) of the form $\bigl(T,(x_{\rho\sigma})\bigr)$.
We say that a stable curve is \emph{smooth} if its underlying rooted tree $T$ has only one interior vertex; we denote a smooth stable curve by the tuple $\bx \in \bC^r$ associated to the root.
\null\hfill$\triangle$
\end{definition}

We now define the notion of a slice of a stable rooted tree.
This can be thought of as a local slice for the action of $G_1$ on the corresponding stratum in $\ol M_r^\bC$, which is necessary data for the chart that we will associate to $T$.

\begin{definition}
\label{def:slice_and_CT}
If $T \in K_r^\bC$ is a stable rooted tree, then a \emph{slice} of $T$ is a pair $\bs = (s_0,s_1)$ of functions $s_i\colon V_\inte(T) \to V(T)$ such that, for every $\rho \in V_\inte(T)$, $s_0(\rho)$ and $s_1(\rho)$ are distinct elements of $\incom(\rho)$.
For a sliced tree $T = (T, \bs)$, we define a map
\begin{align}
C_T
\colon
\prod_{\rho \in V_\inte(T)}
\bigl((\bC\setminus\{0,1\})^{\#\incom(\rho) - 2}\setminus\Delta\bigr)
\to
\ol M_r,
\qquad
(x_{\rho,\sigma})_{{\rho \in V_\inte(T),}
\atop
{\sigma \in \incom(\rho) \setminus \bs(\rho)}}
\mapsto
\bigl(T,(\wt\bx_\rho)\bigr),
\end{align}
where $\wt\bx_\rho$ is defined to be the extension of $\bx_\rho$ by appending $\wt x_{\rho,s_i(\rho)} \coloneqq i$ for $i = 0,1$.
Then $C_T$ restricts to a bijection from its domain to $\ol M_{r,T}$.
\null\hfill$\triangle$
\end{definition}

\noindent
Note that by equation (7) in \cite{b:2ass}, the domain of $C_T$ is contained (as a Zariski-open subset) in $\bC^{d(T)}$.

The polynomials we now define will form the component functions in our charts.

\begin{definition}
\label{def:Mr_local_model}
For $C = \bigl(T, (\bx_\rho)\bigr) \in \SC_r$ a stable curve, we define the following collection of polynomials in variables $b_\upsilon$, $\upsilon \in V_\inte(T) \setminus \{\rho_\root\}$:
\begin{align}
p^C_{\rho\sigma}(\bb)
\coloneqq
\sum_{\tau \in [\rho,\sigma[} \Bigl(x_{\tau\sigma}\prod_{\upsilon \in ]\rho,\tau]} b_\upsilon\Bigr),
\quad
\rho \in V_\inte(T), \sigma \in T_\rho \setminus \{\rho\}.
\end{align}
This allows us to define, for any $T \in K_r^\bC$, the associated local model $X_T$:
\begin{gather}
\label{eq:Kr_local_model_def}
X_T
\coloneqq
\Bigl(
\prod_{\rho \in V_\inte(T)}
\bigl((\bC\setminus\{0,1\})^{\#\incom(\rho)-2}\setminus\Delta\bigr)
\times
\bC^{\#V_\inte(T)-1}
\Bigr)
\setminus
\bigcup_{i < j}
Z(q_{ij})
\subset
\bC^{r-2},
\end{gather}
where the fact that the middle space lies in $\bC^{r-2}$ follows from equation (8) in \cite{b:2ass}.
Denoting an element of $X_T$ by $\bigl((\bx_\rho),\bb\bigr)$, we have set $q_{ij}$ to be the largest polynomial factor of $p_{\rho_\root\lambda_i}^{C_T(\bx_\rho)} - p_{\rho_\root\lambda_j}^{C_T(\bx_\rho)}$ not divisible by a monomial in $\bb$, i.e.\
\begin{align}
p_{\rho_\root\lambda_i}^{C_T(\bx_\rho)}(\bb) - p_{\rho_\root\lambda_j}^{C_T(\bx_\rho)}(\bb)
=
q_{ij}\bigl((\bx_\rho),\bb\bigr)\cdot\!\!\!\prod_{{\rho \in V_\inte(T),}
\atop
{\rho\neq\rho_\root}}
b_\rho^{a_\rho},
\qquad
b_\rho \nmid q_{ij} \:\forall\: \rho \in V_\inte(T)\setminus\{\rho_\root\}.
\end{align}
We define a map $\pi_T\colon \bC^{r-2} \to p_T$ like so:
\begin{align}
\pi_T\bigl((\bx_\rho),\bb\bigr)_\sigma
\coloneqq
\begin{cases}
0, & b_\rho = 0,
\\
1, & \text{ otherwise}.
\end{cases}
\end{align}
Then the local model $X_T$ is stratified as $X_T = \bigcup_{T'\geq T} X_{T,T'}$, where we define $X_{T,T'} \coloneqq \pi_T^{-1}(T')$.
$\null\hfill\triangle$
\end{definition}

In the following lemma, we provide an alternate, stratum-by-stratum formulation of the local model $X_T$.
Specifically, we show that the stratum $X_{T,T'}$ is exactly those choices of $\bigl((\bx_\rho),\bb\bigr)$ with $g_T\bigl(\pi_T\bigl((\bx_\rho),\bb\bigr)\bigr) = T'$ and such that for every $\sigma \in V_\inte(T')$, the special points in the screen corresponding to $\sigma$ are distinct.
\begin{lemma}
\label{lem:X_T_strats}
For every sliced tree $T \in K_r^\bC$ and $T' \geq T$, the following equality holds:
\begin{gather}
\label{eq:X_T_stratum_reform}
X_{T,T'}
=
\biggl\{
\bigl((\bx_\rho),\bb\bigr)
\in
\prod_{\rho \in V_\inte(T)} \bigl((\bC\setminus\{0,1\})^{\#\incom(\rho)-2}\setminus\Delta\bigr)
\times
\bC^{\#V_\inte(T)-1}
\:\bigg|\:
g_T\bigl(\pi_T\bigl((\bx_\rho),\bb\bigr)\bigr) = T',
\:
\eqref{eq:X_T_alternate_form}
\biggr\},
\\
\label{eq:X_T_alternate_form}
p_{\rho\sigma_1}^{C_T(\bx_\rho)}(\bb)
\neq
p_{\rho\sigma_2}^{C_T(\bx_\rho)}(\bb)
\text{ for all }
\rho \in V_\inte(T')
\text{ and distinct }
\sigma_1, \sigma_2 \in \incom_{T'}(\rho).
\end{gather}
\end{lemma}

\begin{proof}
Fix $\rho \in V_\inte(T')$, $\sigma_1,\sigma_2 \in \incom_{T'}(\rho)$ distinct, and $\lambda_i \in T'_{\sigma_i}$, $i \in \{1,2\}$.
To prove \eqref{eq:X_T_stratum_reform}, we will show that $q_{ij}\bigl((\bx_\rho),\bb\bigr)$ and $p_{\rho\sigma_1}^{C_T(\bx_\rho)} - p_{\rho\sigma_2}^{C_T(\bx_\rho)}$ coincide on $X_{T,T'}$.
We begin by simplifying $p_{\rho_\root\lambda_1}^{C_T(\bx_\rho)} - p_{\rho_\root\lambda_2}^{C_T(\bx_\rho)}$:
\begin{align}
p_{\rho_\root\lambda_1}^{C_T(\bx_\rho)}(\bb) - p_{\rho_\root\lambda_2}^{C_T(\bx_\rho)}(\bb)
&=
\sum_{\tau \in [\rho_\root,\lambda_1[} \Bigl(\wt x_{\tau\lambda_1}\prod_{\upsilon \in ]\rho_\root,\tau]} b_\upsilon\Bigr)
-
\sum_{\tau \in [\rho_\root,\lambda_2[} \Bigl(\wt x_{\tau\lambda_2}\prod_{\upsilon \in ]\rho_\root,\tau]} b_\upsilon\Bigr)
\\
&=
\prod_{\upsilon \in ]\rho_\root,\rho]} \!\!\!b_\upsilon\cdot
\biggl(\sum_{\tau \in [\rho,\lambda_1[} \Bigl(\wt x_{\tau\lambda_1}\prod_{\upsilon \in ]\rho,\tau]} b_\upsilon\Bigr)
-
\sum_{\tau \in [\rho,\lambda_2[} \Bigl(\wt x_{\tau\lambda_2}\prod_{\upsilon \in ]\rho,\tau]} b_\upsilon\Bigr)\biggr)
\nonumber
\end{align}
Using the inequality $\wt x_{\rho\lambda_1} = \wt x_{\rho\sigma_1} \neq \wt x_{\rho\sigma_2} = \wt x_{\rho\lambda_2}$, we therefore have the following formula for $q_{ij}$:
\begin{align}
q_{ij}(\bb)
=
\sum_{\tau \in [\rho,\lambda_1[} \Bigl(\wt x_{\tau\lambda_1}\prod_{\upsilon \in ]\rho,\tau]} b_\upsilon\Bigr)
-
\sum_{\tau \in [\rho,\lambda_2[} \Bigl(\wt x_{\tau\lambda_2}\prod_{\upsilon \in ]\rho,\tau]} b_\upsilon\Bigr).
\end{align}
\eqref{eq:in_of_gT} implies that we may rewrite $q_{ij}$ on $X_{T,T'}$ like so:
\begin{align}
\bigl(q_{ij}|_{X_{T,T'}}\bigr)(\bb)
&=
\sum_{\tau \in [\rho,\sigma_1[} \Bigl(\wt x_{\tau\sigma_1}\prod_{\upsilon \in ]\rho,\tau]} b_\upsilon\Bigr)
-
\sum_{\tau \in [\rho,\sigma_2[} \Bigl(\wt x_{\tau\sigma_2}\prod_{\upsilon \in ]\rho,\tau]} b_\upsilon\Bigr)
\\
&=
\bigl((p_{\rho\sigma_1} - p_{\rho\sigma_2})|_{X_{T,T'}}\bigr)(\bb).
\nonumber
\end{align}
This establishes \eqref{eq:X_T_stratum_reform}.
\end{proof}

The following corollary shows that $X_T$ always contains a certain set.
We will use this in our proof of Lemma~\ref{lem:2Mn_sep_proper}, to show that the analytic and Gromov topologies on $\ol M_r^\bC$ coincide.

\begin{corollary}
\label{cor:X_T_contains_0}
For every $T \in K_r^\bC$, $X_T$ is a Zariski-open subset of $\bC^{r-2}$ which contains the set
\begin{align}
\label{eq:set_that_X_T_contains}
\prod_{\rho \in V_\inte(T)} \bigl((\bC\setminus\{0,1\})^{\#\incom(\rho)-2}\setminus\Delta\bigr) \times \{\bzero\}.
\end{align}
\end{corollary}

\begin{proof}
By Def.~\ref{def:Mr_local_model}, $X_T$ is open in $\bC^{r-2}$.
To show that $X_T$ contains \eqref{eq:set_that_X_T_contains}, it follows from Lemma~\ref{lem:X_T_strats} that it is equivalent to show that for every $\rho \in V_\inte(T)$ and $\sigma_1,\sigma_2$ distinct incoming vertices of $\rho$, $p_{\rho\sigma_1}^{C_T(\bx_\rho)}(\bb) - p_{\rho\sigma_2}^{C_T(\bx_\rho)}(\bb)$ has a nonzero constant term when thought of as a function of $\bb$.
This constant term is $\wt x_{\rho\sigma_1} - \wt x_{\rho\sigma_2}$, which is indeed nonzero.
\end{proof}

\begin{definition}
\label{def:Mr_charts}
Fix tree $T \in K_r^\bC$ equipped with a slice $\bs$.
We define $\varphi_T\colon X_T \to \ol M_r^\bC$ like so:
\begin{gather}
\varphi_T\bigl((\bx_\rho),\bb\bigr)
\coloneqq
\bigl(
g_T\bigl(\pi_T\bigl((\bx_\rho),\bb\bigr)\bigr),
\bigl(\bx^{\varphi_T(\bb)}_\rho\bigr)\bigr),
\qquad
\bx^{\varphi_T(\bb)}_\rho
\coloneqq
\Bigl(p_{\rho\sigma}^{C_T(\bx_\rho)}(\bb)\Bigr).
\nonumber
\end{gather}
See Fig.~\ref{fig:Mr_chart} for an illustration of this definition.

Given a sliced tree $(T,\bs)$ and $\boldr \in \{0,1\}^{\#V_\inte(T)-1}$, we can define the \emph{pushforward slice} $\wt\bs$ on $g_T(\boldr)$.
To do so, fix $\rho \in V_\inte(T)$ and $i \in \{0,1\}$.
In the presentation of $\incom_{g_T(\boldr)}(\rho)$ given in \eqref{eq:in_of_gT}, we define $\wt s_i(\rho)$ to be the unique element $\sigma$ of $\incom_{g_T(\boldr)}(\rho)$ satisfying the property that if we denote the path $[\rho,\sigma]_T$ by $\rho=\tau_1,\tau_2,\ldots,\tau_k=\sigma$, then we have $\tau_2 = s_i(\tau_1)$ and $\tau_{j+1} = s_0(\tau_j)$ for $j \in [2,k-1]$.
$\null\hfill\triangle$
\end{definition}

\begin{figure}[H]
\centering
\def\svgwidth{1.0\columnwidth}
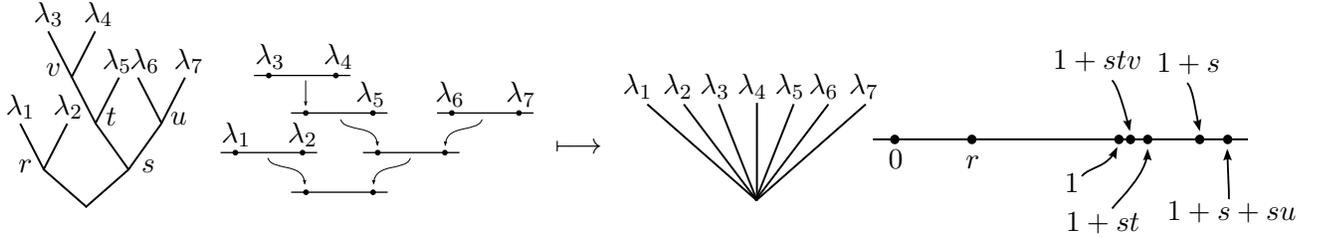
\caption{We illustrate one of the charts on $\ol M_7^\bC$, associated to the stable curve on the left.
The arguments are denoted $r,s,t,u,v$, and in this figure we take them all to be nonzero.
The left resp.\ right marked points in each screen of the stable curve on the left are at 0 resp.\ 1.
\label{fig:Mr_chart}}
\end{figure}

\begin{remark}
We will use all the charts $\varphi_T$, for $T \in K_r^\bC$ equipped with a slice, to construct an atlas on $\ol{2M}_\bn^\bC$.
In fact, one could construct a smaller atlas, which consists of a chart for every dimension-0 element of $K_r^\bC$, where we make a choice of a slice for each such tree.
\null\hfill$\triangle$
\end{remark}

\begin{example}
In the following figure, we depict one of the charts on $\ol M_r^\bC$, associated to the 0-dimensional stable curve on the left.
\begin{figure}[H]
\centering
\def\svgwidth{0.8\columnwidth}
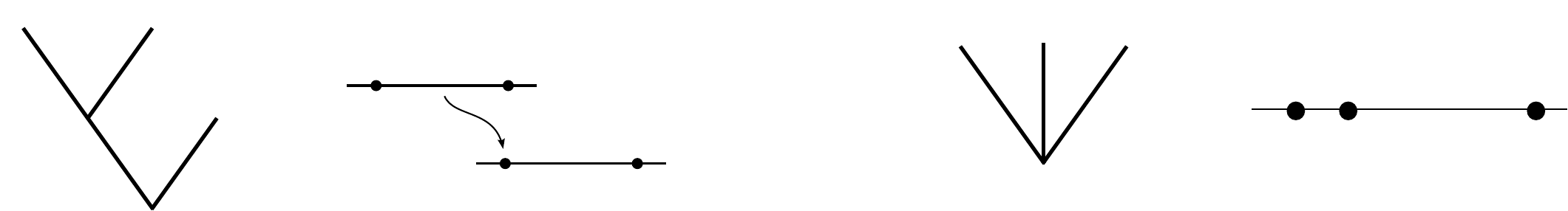
\end{figure}
\noindent
The domain of this chart is $\bC \setminus \{1\}$.
We are therefore defining an atlas on $\ol M_3 \simeq \bCP^1$ consisting of four charts (ignoring choice of slice, which does not affect the chart's image): $\bCP^1 \setminus \{1,\infty\}$, $\bCP^1 \setminus \{0,\infty\}$, $\bCP^1 \setminus \{0,1\}$, and $\bCP^1 \setminus \{0,1,\infty\}$.
The first three are the charts associated to the 0-dimensional elements of $K_r^\bC$, and the last chart is associated to $T_3^\top$.
\null\hfill$\triangle$
\end{example}

\begin{lemma}
\label{lem:image_of_Kr_chart}
For any sliced tree $T$, $\varphi_T$ is injective, with image $\varphi_T(X_T) = \bigcup_{T' \geq T} \ol M_{r,T'}$.
\end{lemma}

\begin{proof}
{\bf Step 1:}
{\it For $d(T) = 0$, $\varphi_T$ restricts to a bijection from $X_T \cap (\bC\setminus\{0\})^{r-2}$ to the top stratum $M_r^\bC \coloneqq \ol M_{r,T_r^\top}$ of $\ol M_r$.}

\medskip

\noindent
Denote by $\bs$ the slice associated to $T$, and by $\wt\bs$ the pushforward of $\bs$ to $T_r^\top$.
We will prove the statement of Step 1 by showing that $C_{(T_r^\top,\wt\bs)}^{-1} \circ \varphi_T$ restricts to a bijection from $X_T \cap (\bC\setminus\{0\})^{r-2}$ to $(\bC\setminus\{0,1\})^{r-2} \setminus \Delta$.
Clearly we have $\bigl(C_{(T_r^\top,\wt\bs)}^{-1} \circ \varphi_T\bigr)\bigl(X_T \cap (\bC\setminus\{0\})^{r-2}\bigr) \subset (\bC\setminus\{0,1\})^{r-2} \setminus \Delta$.
In the remainder of this step, we will show that this restriction is invertible.

For any $\bb \in X_T \cap (\bC\setminus\{0\})^{r-2}$, $\varphi_T(\bb)$ lies in the open stratum $\ol M_{r,T_r^\top}$, so we can represent $\varphi_T(\bb)$ by the tuple associated to the root.
This tuple takes the following form:
\begin{align}
\label{eq:Mr_gluing_special_form}
\varphi_T(\bb)
=
\biggl(\sum_{{\sigma \in [\rho_\root,\lambda_i[,}
\atop
{x_{\sigma\lambda_i} = 1}} \prod_{\tau \in ]\rho_\root,\sigma]} b_\tau\biggr)_{1\leq i\leq r}
\eqqcolon
\bigl(x'_i(\bb)\bigr)_{1\leq i\leq r}.
\end{align}
Assuming without loss of generality that $\wt s_j(\rho_\root) = \lambda_{r-j}$, we have $x'_{r-j}(\bb) \equiv j$ by the definition of the pushforward slice.
Therefore we have the following formula for the map $C_{(T_r^\top,\wt \bs)}^{-1} \circ \varphi_T$:
\begin{align}
\bigl(C_{(T_r^\top,\wt \bs)}^{-1} \circ \varphi_T\bigr)(\bb)
=
\biggl(\sum_{{\sigma \in [\rho_\root,\lambda_i[,}
\atop
{x_{\sigma\lambda_i} = 1}} \prod_{\tau \in ]\rho_\root,\sigma]} b_\tau\biggr)_{1\leq i\leq r-2}
=
\bigl(x_i'(\bb)\bigr)_{1 \leq i \leq r-2}.
\end{align}
We will now show that for any $\ba \in (\bC\setminus\{0,1\})^{r-2} \setminus \Delta$, we can solve the collection of equations
\begin{align}
\label{eq:eqns_to_invert_injectivity_for_Mr}
x_i'(\bb) = \ba,
\quad
1 \leq i \leq r-2
\end{align}
for $\bb$, which is tantamount to showing that the restriction of $C_{(T_r^\top,\wt\bs)}^{-1} \circ \varphi_T$ is invertible.
We will do so by induction.
For $r=2$, it is trivial to solve \eqref{eq:eqns_to_invert_injectivity_for_Mr}.
Next, suppose that we have proven that we can solve \eqref{eq:eqns_to_invert_injectivity_for_Mr} up to, but not including, some $r \geq 3$, and consider a sliced tree $T \in K_r^\bC$ and the associated set of equations as in \eqref{eq:eqns_to_invert_injectivity_for_Mr}.
We may assume without loss of generality that $\lambda_{r-2} = s_1(\rho)$, where we denote by $\rho_0$ the outgoing neighbor of $\lambda_{r-2}$.
Then the system of equations \eqref{eq:eqns_to_invert_injectivity_for_Mr}' obtained by removing the $i=r-2$ equation from \eqref{eq:eqns_to_invert_injectivity_for_Mr} is equal to the system of equations associated to the tree $T'$, where $T'$ is the result of removing the two incoming edges of $\rho_0$.
By induction, we can solve \eqref{eq:eqns_to_invert_injectivity_for_Mr}' for $\bb' \coloneqq (b_\rho)_{\rho \neq \rho_0}$.
It remains to solve the $i=r-2$ equation in \eqref{eq:eqns_to_invert_injectivity_for_Mr}, i.e.\ the equation
\begin{align}
\label{eq:last_eqn_inj_for_Mr}
\sum_{{\sigma \in [\rho_\root,\lambda_{r-2}[,}
\atop
{x_{\sigma\lambda_{r-2}} = 1}} \prod_{\tau \in ]\rho_\root,\sigma]} b_\tau
=
a_{r-2}.
\end{align}
The variable $b_{\rho_0}$ occurs only once in \eqref{eq:last_eqn_inj_for_Mr}, in the monomial associated to $\sigma = \rho_0$.
We can therefore solve \eqref{eq:last_eqn_inj_for_Mr} for $b_{\rho_0}$:
\begin{align}
\label{eq:solved_for_last_var_inj_of_Mr}
b_{\rho_0}
=
\frac
{a_{r-2}
-
\sum_{{\sigma \in [\rho_\root,\rho_0[,}
\atop
{x_{\sigma\lambda_{r-2}}=1}}
\prod_{\tau \in ]\rho_\root,\sigma]} b_\tau}
{\prod_{\tau \in ]\rho_\root,\rho_0[} b_\tau}.
\end{align}
Our final task is to justify why \eqref{eq:solved_for_last_var_inj_of_Mr} produces a well-defined $b_{\rho_0}$ with the property that $\bb = \bb' \cup \{b_{\rho_0}\}$ lies in $X_T \cap (\bC\setminus\{0\})^{r-2}$.
By induction, $\bb'$ lies in $X_{T'} \cap (\bC\setminus\{0\})^{r-3}$, so the denominator of the right-hand side of \eqref{eq:solved_for_last_var_inj_of_Mr} is nonzero.
With $b_{\rho_0}$ chosen as in \eqref{eq:solved_for_last_var_inj_of_Mr}, we have $\bb \in X_T$: indeed, if $\bb$ were to lie in $\bC^{r-2} \setminus X_T$, then $\bigl(x_i'(\bb)\bigr)$ would not lie in $(\bC\setminus\{0,1\})^{r-2} \setminus \Delta$, which would contradict our assumption.
Finally, we have $\bb \in (\bC\setminus\{0\})^{r-2}$: by the inductive hypothesis, $\bb'$ lies in $(\bC\setminus\{0\})^{r-3}$, and we must have $b_{\rho_0} \neq 0$, since otherwise \eqref{eq:solved_for_last_var_inj_of_Mr} would imply $a_{r-1}=a_{r-2}$.

\medskip

\noindent
{\bf Step 2:}
{\it We prove the lemma in the $d(T) = 0$ case.}

\medskip

\noindent
We will establish Step 2 by showing that for any $T' \geq T$, $\varphi_T$ restricts to a bijection from $X_{T,T'}$ to $\ol M_{r,T'}^\bC$.
Define $\boldr \in \{0,1\}^{\#(V_\inte(T)\setminus\{\rho_\root\})}$ such that $T' = g_T(\boldr)$.
We can decompose $T$ as $T = \bigcup_{\rho \in V_\inte(T')} T(\rho)$ like so: for any $\rho \in V_\inte(T')$, denote its incoming neighbors by $\sigma_1, \ldots, \sigma_k$, and regard $\rho, \sigma_1, \ldots, \sigma_k$ as elements of $V(T)$ as in \eqref{eq:in_of_gT}.
Then we define $T(\rho)$ to be the subtree of $T$ bounded by $\rho, \sigma_1, \ldots, \sigma_k$.
In addition, we denote by $s_\rho$ the number of leaves of $T(\rho)$.
This decomposition of $T$ induces the following decomposition of $X_{T,T'}$:
\begin{align}
X_{T,T'}
=
\prod_{\rho \in V_\inte(T')}
X_{T(\rho),T_{s_\rho}^\top}.
\end{align}
With respect to this decomposition, $\varphi_T$ decomposes in turn as $\varphi_T = \bigl(\varphi_{T(\rho)}\bigr)_{\rho \in V_\inte(T)}$.
The current step now follows from Step 1, since each $T(\rho)$ has $d\bigl(T(\rho)\bigr) = 0$.

\medskip

\noindent
{\bf Step 3:}
{\it We prove the lemma.}

\medskip

\noindent
To show that the general case of the lemma holds, we can make an argument similar to our proof of the $d(T) = 0$ case in Steps 1 and 2.
We will now summarize how this goes.
First, consider the analogue of Step 1: the statement that for any sliced tree $T$, $\varphi_T$ restricts to the following bijection:
\begin{align}
\label{eq:restricted_gT_bijection_general_case}
\varphi_T
\colon
X_T^{\neq0}
\prod_{\rho \in V_\inte(T)}
\bigl((\bC\setminus\{0,1\})^{\#\incom(\rho)-2}\setminus\Delta\bigr)
\times
(\bC\setminus\{0\})^{\#V_\inte(T)-1}
\quad
\sr{\simeq}{\lra}
\quad
M_r^\bC.
\end{align}
We do so by showing that the composition of this restriction with $C_{(T_r^\top,\wt \bs)}^{-1}$ is a bijection, where $\wt\bs$ is the pushforward slice.
We have the following formula for this composition:
\begin{align}
\bigl(C_{(T_r^\top,\wt \bs)}^{-1} \circ \varphi_T\bigr)\bigl((\bx_\rho),\bb\bigr)
=
\biggl(\sum_{\sigma \in [\rho_\root,\lambda_i[} \Bigl(x_{\tau\sigma}\prod_{\tau \in ]\rho_\root,\sigma]} b_\tau\Bigr)\biggr)_{1\leq i\leq r-2}
=
\bigl(x_i'\bigl((\bx_\rho),\bb\bigr)\bigr)_{1 \leq i \leq r-2}.
\end{align}
To establish that the composition in question is a bijection, we must invert the system of equations
\begin{align}
\label{eq:eqns_to_invert_injectivity_for_Mr_general}
x_i'\bigl((\bx_\rho),\bb\bigr) = a_i,
\quad
1 \leq i \leq r-2,
\end{align}
for $\ba \in (\bC \setminus \{0,1\})^{r-2} \setminus \Delta$.
We can show that \eqref{eq:eqns_to_invert_injectivity_for_Mr_general} can be solved for $\bigl((\bx_\rho),\bb\bigr)$ by induction on $d(T)$.
We established the $d(T) = 0$ case in Step 1.
To establish the inductive hypothesis, suppose that we wish to solve \eqref{eq:eqns_to_invert_injectivity_for_Mr_general} for some particular $T$ with $d(T) \geq 1$.
Without loss of generality, we may assume $\lambda_{r-2} \neq \wt s_i(\rho_0)$, where $\rho_0$ denotes the outgoing neighbor of $\lambda_{r-2}$.
By induction, we can invert the system of equations \eqref{eq:eqns_to_invert_injectivity_for_Mr_general}' resulting from removing the $i=r-2$ equation from \eqref{eq:eqns_to_invert_injectivity_for_Mr_general}; it remains to show that we can solve the equation $x_{r-2}'\bigl((x_\rho),\bb\bigr) = a_{r-2}$.
We can do so, by the following formula, in which we denote by $\rho_0$ the outgoing neighbor of $\lambda_{r-2}$:
\begin{align}
x_{\rho_0\lambda_{r-2}}
=
\frac
{a_{r-2}
-
\sum_{\sigma \in [\rho_\root,\rho_0[}
\bigl(x_{\tau\sigma}
\prod_{\tau \in ]\rho_\root,\sigma]} b_\tau\bigr)}
{\prod_{\tau \in ]\rho_\root,\rho_0]} b_\tau}.
\end{align}
This establishes that $\varphi_T$ restricts to a bijection as in \eqref{eq:restricted_gT_bijection_general_case}.
The remainder of the current step follows as in Step 2.
\end{proof}

\begin{lemma}
\label{lem:Mr_transitions_algebraic}
For any sliced trees $(T_1,\bs^1), (T_2,\bs^2) \in \ol M_r^\bC$, the transition map $\varphi_{T_2}^{-1} \circ \varphi_{T_1}\colon X_{T_1,T_r^\top} \to X_{T_2,T_r^\top}$ is a morphism.
\end{lemma}

\begin{proof}
Denote by $\wt\bs^1, \wt\bs^2$ the pushforwards to $T_r^\top$ of $\bs^1, \bs^2$.

\medskip

\noindent
{\bf Step 1:}
{\it We show that the map $C_{(T_r^\top,\wt\bs^2)}^{-1} \circ \varphi_{(T_1,\bs^1)}$ is a morphism from $X_{T_1,T_r^\top}$ to $(\bC\setminus\{0,1\})^{r-2}\setminus\Delta$.}

\medskip

\noindent
The putative chart $\varphi_{(T_1,\bs^1)}$ acts in the following way, where as usual we represent an element of $M_r$ by the tuple $\bx_{\rho_\root}$:
\begin{align}
\label{eq:Mr_transition_first_tuple}
\varphi_{(T_1,\bs^1)}\bigl((\bx_\rho),\bb\bigr)
=
\Bigl(p_{\rho_\root\lambda_i}^{C_{(T_1,\bs^1)}(\bx_\rho)}\Bigr)_{1 \leq i \leq r}.
\end{align}
Without loss of generality, suppose $\wt s^2_j = r - j$ for $j \in \{0,1\}$.
The element of $M_r$ in \eqref{eq:Mr_transition_first_tuple} is equivalent to the element
\begin{align}
\Biggl(
\frac
{p_{\rho_\root\lambda_i}^{C_{(T_1,\bs^1)}(\bx_\rho)} - p_{\rho_\root\lambda_r}^{C_{(T_1,\bs^1)}(\bx_\rho)}}
{p_{\rho_\root\lambda_{r-1}}^{C_{(T_1,\bs^1)}(\bx_\rho)} - p_{\rho_\root\lambda_r}^{C_{(T_1,\bs^1)}(\bx_\rho)}}
\Biggr)_{1 \leq i \leq r}.
\end{align}
This tuple has the property that the $r$-th resp.\ $(r-1)$-th elements are identically 0 resp.\ 1, so we can now deduce a formula for the composition $C_{T_r^\top,\wt\bs^2}^{-1} \circ \varphi_{(T_1,\bs^1)}$:
\begin{align}
\bigl(C_{T_r^\top,\wt\bs^2}^{-1} \circ \varphi_{(T_1,\bs^1)}\bigr)
\bigl((\bx_\rho),\bb\bigr)
=
\Biggl(
\frac
{p_{\rho_\root\lambda_i}^{C_{(T_1,\bs^1)}(\bx_\rho)} - p_{\rho_\root\lambda_r}^{C_{(T_1,\bs^1)}(\bx_\rho)}}
{p_{\rho_\root\lambda_{r-1}}^{C_{(T_1,\bs^1)}(\bx_\rho)} - p_{\rho_\root\lambda_r}^{C_{(T_1,\bs^1)}(\bx_\rho)}}
\Biggr)_{1 \leq i \leq r-2}.
\end{align}
This expression is a rational function in $\bigl((\bx_\rho),\bb\bigr)$, so we have established Step 1.

\medskip

\noindent
{\bf Step 2:}
{\it We complete the proof of the lemma.}

\medskip

\noindent
First, note that an argument similar to the ones made in Lemma~\ref{lem:image_of_Kr_chart} shows that the map $\varphi_{(T_2,\bs^2)}^{-1} \circ C_{(T_r^\top,\wt\bs^2)}$ is a morphism from $(\bC\setminus\{0,1\})^{r-2}\setminus\Delta$ to $X_{T_2,T_r^\top}$.
It follows from this and Step 1 that $\varphi_{(T_2,\bs^2)}^{-1} \circ \varphi_{(T_1,\bs^1)}$ is a morphism from $X_{T_1,T_r^\top}$ to $X_{T_2,T_r^\top}$.
As illustrated in Example~\ref{ex:Mr_morphism}, this morphism extends to a morphism from $X_{T_1}$ to $X_{T_2}$.
\end{proof}

\begin{example}
\label{ex:Mr_morphism}
In the following figure, we compute two charts $\varphi_{T_1}, \varphi_{T_2}$ on $\ol M_4^\bC$, in the slice on $M_4^\bC$ pushed forward from $T_2$:
\begin{figure}[H]
\centering
\def\svgwidth{1.0\columnwidth}
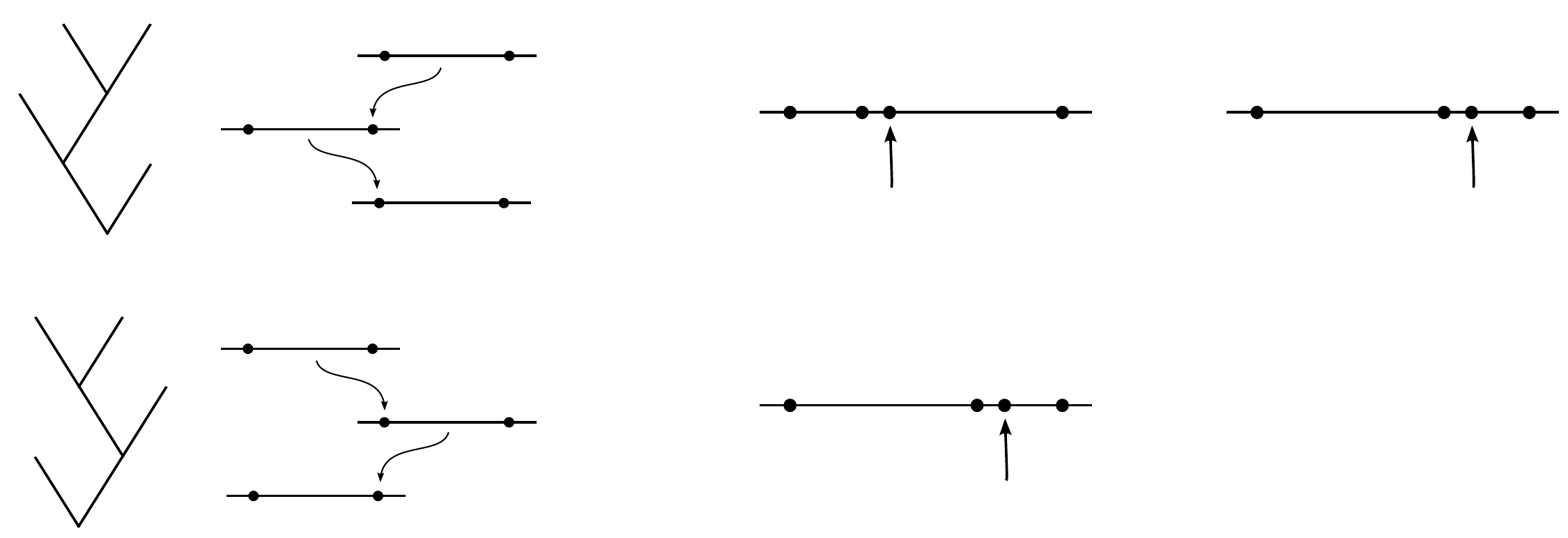
\end{figure}
\noindent
We see that to compute the transition map $\varphi_{T_2}^{-1} \circ \varphi_{T_1}$ on the open locus $X_{T_1,T_4^\top}$, we must solve the equations $1+r's' = 1+s$, $1 + r' = 1/r$.
These yield the formula
\begin{align}
\bigl(\varphi_{T_2}^{-1} \circ \varphi_{T_1}\bigr)|_{X_{T_1,T_4^\top}}(r,s)
=
\Bigl(\frac{1-r}r, \frac{rs}{1-r}\Bigr).
\end{align}
Denote by $T$ the common result of gluing either $T_1$ or $T_2$ along the edge labeled by $r$ resp.\ $r'$.
Our formula for $\varphi_{T_2}^{-1} \circ \varphi_{T_1}$ on $X_{T_1,T_r^\top}$ extends to $X_{T_1,T}$ by setting $s=0$:
\begin{align}
\bigl(\varphi_{T_2}^{-1}\circ\varphi_{T_1}\bigr)|_{X_{T_1,T}}(r,0)
=
\Bigl(\frac{1-r}r, 0\Bigr).
\end{align}
\null\hfill$\triangle$
\end{example}

\subsection{We equip $\ol M_r^\bC$ with the structure of a proper abstract variety}
\label{ss:Mr-bar-charts}

In this subsection, we will use the atlas we constructed in \S\ref{ss:Mr_charts} to equip $\ol M_r$ with the structure of an abstract prevariety, which we will then show is separated (hence a variety) and proper.

\medskip

\noindent
{\bf Definitions 1.1 and 3.5, \cite{osserman:atlas}.}
A \emph{prevariety} $X$ over an algebraically closed field $k$ is an irreducible topological space together with an open cover $U_1,\ldots,U_m$ and a collection of homeomorphisms $\varphi_i\colon X_i \stackrel{\simeq}{\longrightarrow} U_i$, where each $X_i \subset \bA^{n_i}$ is an affine variety equipped with the Zariski topology, and we require that every transition map
\begin{align}
\varphi_{ij}
\coloneqq
\varphi_j^{-1}\circ\varphi_i,
\qquad
\varphi_{ij}
\colon
\varphi_i^{-1}(U_i\cap U_j)
\to
\varphi_j^{-1}(U_i\cap U_j)
\end{align}
is a morphism.
We say that a prevariety $X$ is a \emph{variety} if the image of the diagonal morphism $X \to X \times X$ is closed.
$\null\hfill\triangle$

\begin{lemma}
\label{lem:Mr_variety}
There is a unique topology on $\ol M_r^\bC$ such that $\bigl(\ol M_r^\bC, (\varphi_T)\bigr)$ is a variety.
\end{lemma}

\begin{proof}
{\bf Step 1:}
{\it We equip $\ol M_r^\bC$ with a topology in which the charts are homeomorphisms, and observe that this is the unique topology with this property.}

\medskip

\noindent
Define a topology on $\ol M_r^\bC$ by declaring $U \subset \ol M_r^\bC$ to be open exactly when $\varphi_T^{-1}(U) \subset X_T$ is Zariski-open for every $T$.
The continuity of the transition maps then implies that every chart $\varphi_T\colon X_T \to \ol M_r^\bC$ is a homeomorphism onto its image $\varphi_T(X_T) \eqqcolon U_C$.
It is clear that this is the unique topology with this property.

\medskip

\noindent
{\bf Step 2:}
{\it The top stratum $M_r^\bC \subset \ol M_r^\bC$ is homeomorphic to an open subset of $\bC^{r-2}$, and it is dense in $\ol M_r^\bC$.}

\medskip

\noindent
By Lemma~\ref{lem:image_of_Kr_chart} and Step 1 of the current lemma, the top stratum $M_r^\bC = \ol M_{r,T_r^\top}^\bC$ is homeomorphic to $X_{T,T_r^\top}$.
By \eqref{eq:Kr_local_model_def}, the latter space is an open subset of affine space:
\begin{align}
X_{T,T_r^\top}
=
(\bC^r \setminus \Delta)/_{\bC \rtimes (\bC\setminus\{0\})}
=
(\bC \setminus \{0,1\})^{r-2} \setminus \Delta
\subset
\bC^{r-2}.
\end{align} 

To see that $M_r^\bC$ is dense in $\ol M_r^\bC$, it suffices to show that, for every $T \in K_r^\bC$, $\varphi_T^{-1}(M_r^\bC)$ is dense in $X_T$.
This follows from the fact that $\varphi_T^{-1}(M_r^\bC)$ is equal to $X_T^{\neq0}$, where this latter object was defined in \eqref{eq:restricted_gT_bijection_general_case}.

\medskip

\noindent
{\bf Step 3:}
{\it We complete the proof of the lemma.}

\medskip

\noindent
Lemma~\ref{lem:Mr_transitions_algebraic} shows that the transition maps are morphisms, so to complete the proof of the claim, it only remains to show that $\ol M_r^\bC$ is irreducible.
Suppose that $K,L \subset \ol M_r^\bC$ are nonempty subsets with $K \cup L = \ol M_r^\bC$.
Then $K \cap M_r^\bC$ and $L \cap M_r^\bC$ are relatively closed subsets of $M_r^\bC$ whose union is $M_r^\bC$, so by Step 2 and the fact that open subsets of complex affine space are irreducible, either $K$ or $L$ contains $M_r^\bC$.
Suppose without loss of generality that $K$ contains $M_r^\bC$.
It follows from the density of $M_r^\bC$ in $\ol M_r^\bC$ that $K$ is equal to $\ol M_r^\bC$, so we have proven the irreducibility of $\ol M_r^\bC$.
\end{proof}

Next, we will show that this topology on $\ol M_r^\bC$ is separated and proper.
We first recall characterizations of separatedness and properness of a complex scheme in terms of the analytic topology.

\medskip

\noindent
{\bf Corollary 3.3, \cite{osserman:complex}.}
{\it Let $X$ be a scheme of finite type over $\bC$.
Then $X$ is separated over $\bC$ if and only if $X_\an$ is Hausdorff.}

\medskip

\noindent
{\bf Corollary 3.5, \cite{osserman:complex}.}
{\it A separated scheme $X$ of finite type over $\bC$ is proper over $\bC$ if and only if $X_\an$ is compact.}

\medskip

Next, we note that we can equip $\ol M_r^\bC$ with an \emph{a priori} different topology, the \emph{Gromov topology}, in which the convergent sequences are exactly the Gromov-convergent sequences.
This is entirely analogous to the construction of the Gromov topology on $\ol{2M}_\bn^\bR$, in which the convergent sequences are exactly the Gromov-convergent ones; the latter notion is defined in \cite[Def.\ 2.3]{b:realization}.
Using the following consequence of the closed mapping theorem, we will show in our proof of Lemma~\ref{lem:Mr_sep_proper} below that the Gromov topology agrees with the analytic topology associated to the variety structure on $\ol M_r^\bC$.
As in \cite[Thm.\ 1.1]{b:realization}, the Gromov topology on $\ol M_r^\bC$ is metrizable, hence first-countable and Hausdorff.

\medskip

\noindent
{\bf Lemma 4.25(d), \cite{lee:top_man}.}
{\it If $F\colon X \to Y$ is a continuous bijection, if $X$ is compact, and if $Y$ is Hausdorff, then $F$ is a homeomorphism.}

\begin{lemma}
\label{lem:Mr_sep_proper}
$\ol M_r^\bC$ is separated and proper.
\end{lemma}

\begin{proof}
{\bf Step 1:}
{\it If $(p^\nu)$ converges to $p^\infty$ in the analytic topology on $\ol M_r^\bC$, then it also converges to $p^\infty$ in the Gromov topology on $\ol M_r^\bC$.}

\medskip

\noindent
Fix such a sequence $p^\nu \to p^\infty$, and let $T$ be the rooted tree such that $p^\infty$ is an element of $\ol M_{r,T}^\bC$.
Without loss of generality we may assume that $(p^\nu)$ is contained in a single stratum $\ol M_{r,T'}^\bC$ with $T' \geq T$.
For simplicity, we will assume $T' = T_r^\top$ for the rest of this step.
Set $\bigl((\bx_\rho^\nu), \bb^\nu)\bigr) \coloneqq \varphi_T^{-1}(p^\nu)$, and note that each $(\bx_\rho^\nu)$ converges to some $\bx_\rho^\infty$, and $(\bb^\nu)$ is a sequence of elements of $(\bC\setminus\{0\})^{\#V_\inte(T)-1}$ that converges to $\bzero$.

To show that $\varphi_T\bigl((\bx_\rho^\nu),\bb^\nu\bigr)$ Gromov-converges to $\varphi_t\bigl((\bx_\rho^\infty),\bzero\bigr)$, we must produce a sequence of rescalings $(\varphi_\rho^\nu) \subset \bC \rtimes (\bC\setminus\{0\})$ for every $\rho \in V_\inte(T)$.
We define $\varphi_\rho^\nu$ by the following formula:
\begin{align}
(\varphi_\rho^\nu)^{-1}(a)
\coloneqq
\frac
{a-\sum_{\tau \in [\rho_\root,\rho[} \Bigl(x_{\tau\rho}\prod_{\upsilon \in ]\rho_\root,\tau]} b_\upsilon\bigr)}
{\prod_{\upsilon \in ]\rho_\root,\rho]} b_\upsilon}.
\end{align}
Let us now show that with this sequence of rescalings, {\sc (special point')} holds.
To do so, we must calculate $\lim_{\nu\to\infty} (\varphi_\rho^\nu)^{-1}(x^\nu_{\rho\lambda_i})$, where we have set $\bigl(T_r^\top,(\bx_\rho^\nu)\bigr) \coloneqq \varphi_T\bigl((\bx_\rho^\nu),\bb^\nu\bigr)$.

\begin{itemize}
\item First, consider the case that $\lambda_i$ lies in $T_\rho$.
We begin by rewriting $x_{\rho\lambda_i}^\nu$:
\begin{align}
x_{\rho\lambda_i}^\nu
&=
\sum_{\tau \in [\rho_\root,\lambda_i[}
\Bigl(x^\nu_{\tau\lambda_i}\prod_{\upsilon \in ]\rho_\root,\tau]} y^\nu_\upsilon\Bigr)
\\
&=
\sum_{\tau \in [\rho_\root,\rho[}
\Bigl(x^\nu_{\tau\rho}\prod_{\upsilon \in ]\rho_\root,\tau]} y^\nu_\upsilon\Bigr)
+
\prod_{\upsilon \in ]\rho_\root,\rho]} y^\nu_\upsilon
\cdot
\sum_{\tau \in [\rho,\lambda_i[} \Bigl(x^\nu_{\tau\lambda_i}\prod_{\upsilon \in ]\rho,\tau]} y^\nu_\upsilon\Bigr)
\nonumber
\end{align}
This allows us to calculate $\lim_{\nu \to \infty} (\varphi_\rho^\nu)^{-1}(x_{\rho\lambda_i}^\nu)$:
\begin{gather}
(\varphi_\rho^\nu)^{-1}(x^\nu_{\rho\lambda_i})
=
\sum_{\tau \in [\rho,\lambda_i[} \Bigl( x^\nu_{\tau\lambda_i}\prod_{\upsilon \in ]\rho,\tau]} y^\nu_\upsilon\Bigr)
\:\:\sr{\nu\to\infty}{\lra}\:\:
x_{\rho\lambda_i}^\infty.
\end{gather}

\item Second, consider the case that $\lambda_i$ does not lie in $T_\rho$, and denote by $\sigma$ the vertex closest to $\rho_\root$ in the path from $\rho$ to $\lambda_i$.
A manipulation similar to the one we made in the previous bullet shows that $(\varphi_\rho^\nu)^{-1}(x_{\rho\lambda_i}^\nu)$ is a Laurent polynomial in $\bb^\nu$ whose lowest-degree term is $(x_{\sigma\lambda_i}^\nu-x_{\sigma\rho}^\nu)/\prod_{\upsilon \in ]\rho_\root,\sigma]} b_\upsilon$.
This term dominates as $\nu\to\infty$, hence in this case we have $\lim_{\nu\to\infty} (\varphi_\rho^\nu)^{-1}(x_{\rho\lambda_i}^\nu) = \infty$.
\end{itemize}

A similar calculation shows that with this choice of rescalings, the {\sc (rescaling)} axiom also holds.

\medskip

\noindent
{\bf Step 2:}
{\it The identity map $\id\colon \ol M_{r,\an}^\bC \to \ol M_{r,\Gr}^\bC$ is continuous.}

\medskip

\noindent
First, note that $\ol M_{r,\an}^\bC$ is first-countable: indeed, fix $x \in \ol M_{r,\an}^\bC$, and choose a chart $\varphi_T$ with $x$ in its image.
The domain $X_T$ of $\varphi_T$ is first-countable, since it is a subspace of $\bC^{r-2}$.
If we choose $(B_i)$ to be the result of pushing forward a neighborhood basis for $\varphi^{-1}(x)$ to $\ol M_{r,\an}^\bC$, then $(B_i)$ is a neighborhood basis for $x$: indeed, this follows from the fact that for any open $U \subset \ol M_{r,\an}^\bC$, $\varphi^{-1}(U)$ is open in $X_T$.

Since $\ol M_{r,\Gr}^\bC$ is first-countable, it follows from \cite[Cor.\ 10.5.c]{willard} that to establish the current step, it suffices to prove that if $(p_i)$ converges to $p_\infty$ in the analytic topology on $\ol M_r^\bC$, then it also converges to $p_\infty$ in the Gromov topology.
This is exactly what we showed in Step 1.

\medskip

\noindent
{\bf Step 3:}
{\it The topological spaces $\ol M_{r,\an}^\bC$ and $\ol M_{r,\Gr}^\bC$ are homeomorphic, where $\ol M_{r,\an}^\bC$ denotes the analytic topology on $\ol M_r^\bC$.}

\medskip

\noindent
For any $T$, it follows from Cor.~\ref{cor:X_T_contains_0} that $X_T$ is a Zariski-open subset of affine space that contains 0.
It follows that there exists an $\eps_T>0$ with $\ol B_{\eps_T}(0) \subset X_T$.
The closed ball $\ol B_{\eps_T}(0)$ is compact in the analytic topology, and $\ol M_{r,\Gr}^\bC$ is Hausdorff, so by Lee's Lemma 4.25(d) quoted above, $I$ restricts to a homeomorphism from $\varphi_T\bigl(\ol B_{\eps_T}(0)\bigr) \subset \ol M_{r,\an}^\bC$ to the same set considered as a subset of $\ol M_{r,\an}^\bC$.
The collection $\bigl(\varphi_T(B_{\eps_T}(0))\bigr)_C$ covers $\ol M_r^\bC$, so this establishes that $\ol M_{r,\an}^\bC$ and $\ol M_{r,\Gr}^\bC$ are homeomorphic.

\medskip

\noindent
{\bf Step 4:}
{\it We prove the lemma.}

\medskip

\noindent
By Step 1 of this proof and Corollaries 3.3 and 3.5 from \cite{osserman:complex}, it suffices to show that $\ol M_{r,\Gr}^\bC$ is Hausdorff and compact.
These properties can be established analogously with Theorems 2.10 and 2.6, \cite{b:realization}; we omit the proof.
\end{proof}

\section{Construction of $\ol{2M}_\bn^\bC$ via explicit charts}
\label{s:Wn-bar}

In this section, we will define $\ol{2M}_\bn^\bC$ as a complex variety.
Our first step toward this goal is to construct it as a set, which is a straightforward modification of the construction of $\ol{2M}_\bn^\bR$ in \cite{b:realization}.
Next, we equip $\ol{2M}_\bn^\bC$ with an atlas, which gives $\ol{2M}_\bn^\bC$ the structure of a proper complex variety with toric singularities.
This allows us to prove our main result, Thm.~\ref{thm:2Mn-bar}:

\begin{proof}[Proof of Thm.~\ref{thm:2Mn-bar}]
We show in Lemma~\ref{lem:2Mn_variety} that the atlas consisting of the gluing maps $2\varphi_{2T}$ defined below endows $\ol{2M}_\bn^\bC$ with the structure of a prevariety.
Next, we show in Lemma~\ref{lem:2Mn_sep_proper} that $\ol{2M}_\bn^\bC$ is separated (hence a variety) and proper.
Finally, we show in Lemma~\ref{lem:2X_2T_reduced_normal} that $2X_{2T}$ is reduced and normal.

We can define the forgetful morphism $\pi\colon \ol{2M}_\bn^\bC \to \ol M_r^\bC$ is defined by sending $\bigl(T_b \to T_s, (\bx_\rho), (\bz_\alpha)\bigr)$ to $\bigl(T_s, (\bx_\rho)\bigr)$, in the notation of \S\ref{ss:2Mn-bar_charts}.
The fact that $\pi$ is a morphism is evident from the form of our atlases.
\end{proof}

\subsection{Definition of $W_\bn^\bC$}
\label{ss:Wn}

We begin by defining the poset $W_\bn^\bC$.
This is the analogue of the poset $W_\bn = W_\bn^\bR$ defined by the first author in \cite{b:2ass}, which he called a \emph{2-associahedron}.
The reader should think of $W_\bn^\bC$ as the poset of strata of our compactified moduli space $\ol{2M}_\bn^\bC$.
We will present two equivalent definitions $W_\bn^{\bC,\tree}, W_\bn^{\bC,\br}$ of $W_\bn^\bC$, analogous to our two presentations of $K_r^\bC$ in \S\ref{ss:Kr}.

\begin{definition}
\label{p:2T}
A \emph{$\bC$-type stable tree-pair of type $\bn$} is a datum $2T = T_b \sr{f}{\to} T_s$, with $T_b, T_s, f$ described below:
\begin{itemize}
\item The \emph{bubble tree} $T_b$ is a rooted tree whose edges are either solid or dashed, which we orient toward the root $\alpha_\root^{T_b}$.
We require $T_b$ to satisfy these properties:
	\begin{itemize}
		\item The vertices of $T_b$ are partitioned as $V(T_b) = V_\comp \sqcup V_\seam \sqcup V_\mk$, where:
			\begin{itemize}
				\item Every $\alpha \in V_\comp$ has $\geq 1$ solid incoming edge, no dashed incoming edges, and either a dashed or no outgoing edge.
				We partition $V_\comp \eqqcolon V_\comp^1 \sqcup V_\comp^{\geq2}$ according to the number of incoming edges of a given vertex.
				\item Every $\alpha \in V_\seam$ has $\geq 0$ dashed incoming edges, no solid incoming edges, and a solid outgoing edge.
				\item Every $\alpha \in V_\mk$ has no incoming edges and either a dashed or no outgoing edge.
				$V_\mk$ has $|\bn|$ elements, which are labeled by pairs $(i,j)$ for $1 \leq i \leq r$ and $1 \leq j \leq n_i$.
				We denote by $\mu_{ij}^{T_b}$ the element of $V_\mk$ labeled by $(i,j)$.
			\end{itemize}
		\item ({\sc stability}) If $\alpha$ is a vertex in $V_\comp^1$ and $\beta$ is its incoming neighbor, then $\#\!\incom(\beta) \geq 2$; if $\alpha$ is a vertex in $V_\comp^{\geq2}$ and $\beta_1,\ldots,\beta_\ell$ are its incoming neighbors, then there exists $j$ with $\#\!\incom(\beta_j) \geq 1$.
	\end{itemize}
	\item The \emph{seam tree} $T_s$ is an element of $K_r^\bC$.
	\item The \emph{coherence map} is a map $f\colon T_b \to T_s$ of sets having these properties:
		\begin{itemize}
			\item $f$ sends root to root, and if $\beta \in \incom(\alpha)$ in $T_b$, then either $f(\beta) \in \incom(f(\alpha))$ or $f(\alpha) = f(\beta)$.
			\item $f$ contracts all dashed edges, and every solid edge whose terminal vertex is in $V_\comp^1$.
			\item For any $\alpha \in V_\comp^{\geq2}$, $f$ maps the incoming edges of $\alpha$ bijectively onto the incoming edges of $f(\alpha)$.
			\item $f$ sends $\mu_{ij}^{T_b}$ to $\lambda_i^{T_s}$.
		\end{itemize}
\end{itemize}
We denote by $W_\bn^\bC$\label{p:WnC} the set of isomorphism classes of $\bC$-type stable tree-pairs of type $\bn$.
Here an isomorphism from $T_b \sr{f}{\to} T_s$ to $T_b' \sr{f'}{\to} T_s'$ is a pair of maps $\varphi_b\colon T_b \to T_b'$ and $\varphi_s\colon T_s \to T_s'$ that fit into a commutative square in the obvious way and that respect all the structure of the bubble trees and seam trees.
We upgrade $W_\bn^\bC$ to a poset, by the obvious analogue of the poset structure on $W_\bn^\bC$ presented in \cite[text preceding Ex.\ 3.6 and Def.\ 3.8]{b:2ass}.
\null\hfill$\triangle$
\end{definition}

\begin{definition}
\label{def:2bracket}
A \emph{$\bC$-type 2-bracket of $\bn$} is a pair $\btB = (B, (2B_i))$\label{p:btB} consisting of a 1-bracket $B \subset \{1,\ldots,r\}$ and a subset $2B_i \subset \{1,\ldots,n_i\}$ for every $i \in B$ such that at least one $2B_i$ is nonempty.
We write $\btB' \subset \btB$ if $B' \subset B$ and $2B_i' \subset 2B_i$ for every $i \in B'$, and we define $\pi(B,(2B_i)) \coloneqq B$.
\null\hfill$\triangle$
\end{definition}

\begin{definition}
\label{def:Wn_br}
A \emph{$\bC$-type 2-bracketing of $\bn$} is a pair $(\sB, \stB)$,\label{p:sBstB} where $\sB$ is a 1-bracketing of $r$ and $\stB$ is a collection of 2-brackets of $\bn$ that satisfies these properties:
\begin{itemize}
\item[] {\sc (1-bracketing)} For every $\btB \in \stB$, $\pi(\btB)$ is contained in $\sB$.

\medskip

\item[] {\sc (2-bracketing)}
Suppose that $\btB, \btB'$ are elements of $\stB$, and that for some $i_0 \in \pi(\btB) \cap \pi(\btB')$, the intersection $2B_{i_0} \cap 2B_{i_0}'$ is nonempty.
Then either $\btB \subset \btB'$ or $\btB' \subset \btB$.

\medskip

\item[] {\sc (root and marked points)} $\stB$ contains $(\{1,\ldots,r\},(\{1,\ldots,n_1\},\ldots,\{1,\ldots,n_r\}))$ and every 2-bracket of $\bn$ of the form $(\{i\},(\{j\}))$.
\end{itemize}

\noindent For any $B_0 \in \sB$, write $\stB_{B_0} \coloneqq \{(B,(2B_i)) \in \stB \:|\: B = B_0\}$.

\begin{itemize}
\item[] {\sc (marked seams are unfused)}
\begin{itemize}
\item For any $B_0 \in \sB$ and for any $i \in B_0$, we have $\bigcup_{\btB \in \stB_{B_0}} 2B_i = \{1,\ldots,n_i\}$.

\item For every $\btB \in \stB_{B_0}$ for which there exists $\btB' \in \stB_{B_0}$ with $\btB' \subsetneq \btB$, and for every $i \in B_0$ and $j \in 2B_i$, there exists $\btB'' \in \stB_{B_0}$ with $\btB'' \subsetneq \btB$ and $2B''_i \ni j$.
\end{itemize}
\end{itemize}
We define $W_\bn^{\bC,\br}$ to be the set of 2-bracketings of $\bn$, with the poset structure defined by declaring $(\sB',\stB') < (\sB,\stB)$ if the containments $\sB' \supset \sB$, $\stB' \supset \stB$ hold and at least one of these containments is proper.
\null\hfill$\triangle$
\end{definition}

By an argument analogous to \cite[Proof of Thm.\ 3.17]{b:2ass}, $W_\bn^{\bC,\tree}$ and $W_\bn^{\bC,\br}$ are isomorphic as posets.
We denote this common poset by $W_\bn^\bC \coloneqq W_\bn^{\bC,\tree} = W_\bn^{\bC,\br}$.

Next, we will define local models $2g_{2T}\colon 2p_{2T} \hra W_\bn^\bC$.
For any stable tree-pair $2T \in W_\bn^\bC$, define a poset $2p_{2T}$ like so:
\begin{gather}
2p_{2T}
\coloneqq
\left\{\left.\bigl((q_\alpha)_{\alpha \in V_\comp(T_b)\setminus\{\alpha_\root\}},(r_\rho)_{\rho \in V_\inte(T_s)\setminus\{\rho_\root\}}\bigr)
\in
\{0,1\}^{\#V_\comp(T_b) + \#V_\inte(T_s) - 2}
\:\right|\:
\eqref{eq:Wn_model_coherences}\right\},
\\
\bigl((q_\alpha),(r_\rho)\bigr)
\leq
\bigl((q_\alpha'),(r_\rho')\bigr)
\quad
\iff
\qquad
q_\alpha \leq q_\alpha' \:\: \forall \alpha,
\quad
r_\rho \leq r_\rho' \:\: \forall \rho.
\nonumber
\end{gather}
Here \eqref{eq:Wn_model_coherences} is the following collection of coherences:
\begin{gather}
\prod_{\gamma \in [\alpha_1,\beta)} q_\gamma
=
\prod_{\gamma \in [\alpha_2,\beta)} q_\gamma
\qquad
\forall \: \alpha_1,\alpha_2 \in V_\comp^{\geq2}(T_b),
\beta \in V_\comp^1(T_b)
:
\pi(\alpha_1) = \pi(\alpha_2) = \pi(\beta),
\label{eq:Wn_model_coherences}
\\
r_\rho
=
\prod_{\gamma \in [\alpha,\beta_\alpha)} q_\gamma
\qquad
\forall \rho \in V_\inte(T_s)\setminus\{\rho_\root\},
\alpha \in V_\comp^{\geq2}(T_b)\cap\pi^{-1}\{\rho\},
\nonumber
\end{gather}
where in the second inequality we define $\beta_\alpha$ to be the first element of $V_\comp^{\geq2}(T_b)$ that the path from $\alpha$ to $\alpha_\root$ passes through.

\begin{definition}
For any stable tree-pair $2T$, define a map $2g_{2T}\colon 2p_{2T} \to W_\bn^\bC$ by sending an element $\bigl((q_\alpha),(r_\rho)\bigr) \in 2p_{2T}$ to the tree-pair $2T' = T_b' \to T_s' \in W_\bn^\bC$ defined like so:
\begin{itemize}
\item
We define $T_s'$ using the gluing map $g_{T_s}$, as $T_s' \coloneqq g_{T_s}\bigl((r_\rho)\bigr)$.

\medskip

\item
To define $T_b'$, we think of cutting $T_b$ at each zero in the tuple $(q_\alpha)$.
After doing so, we will have divided $T_b$ into a number of pieces, each of which can be thought of as a tree-pair $\wt{2T} \in W_\bm^\bC$ whose seam tree is the image in $T_s$ of this piece under $\pi$.
Replace each piece by the top tree-pair $T_b^\top \in W_\bm^\bC$.
\null\hfill$\triangle$
\end{itemize}
\end{definition}

\begin{example}
In the following figure, we illustrate the definition of the map $2g_{2T}$.
\begin{figure}[H]
\centering
\def\svgwidth{0.85\columnwidth}
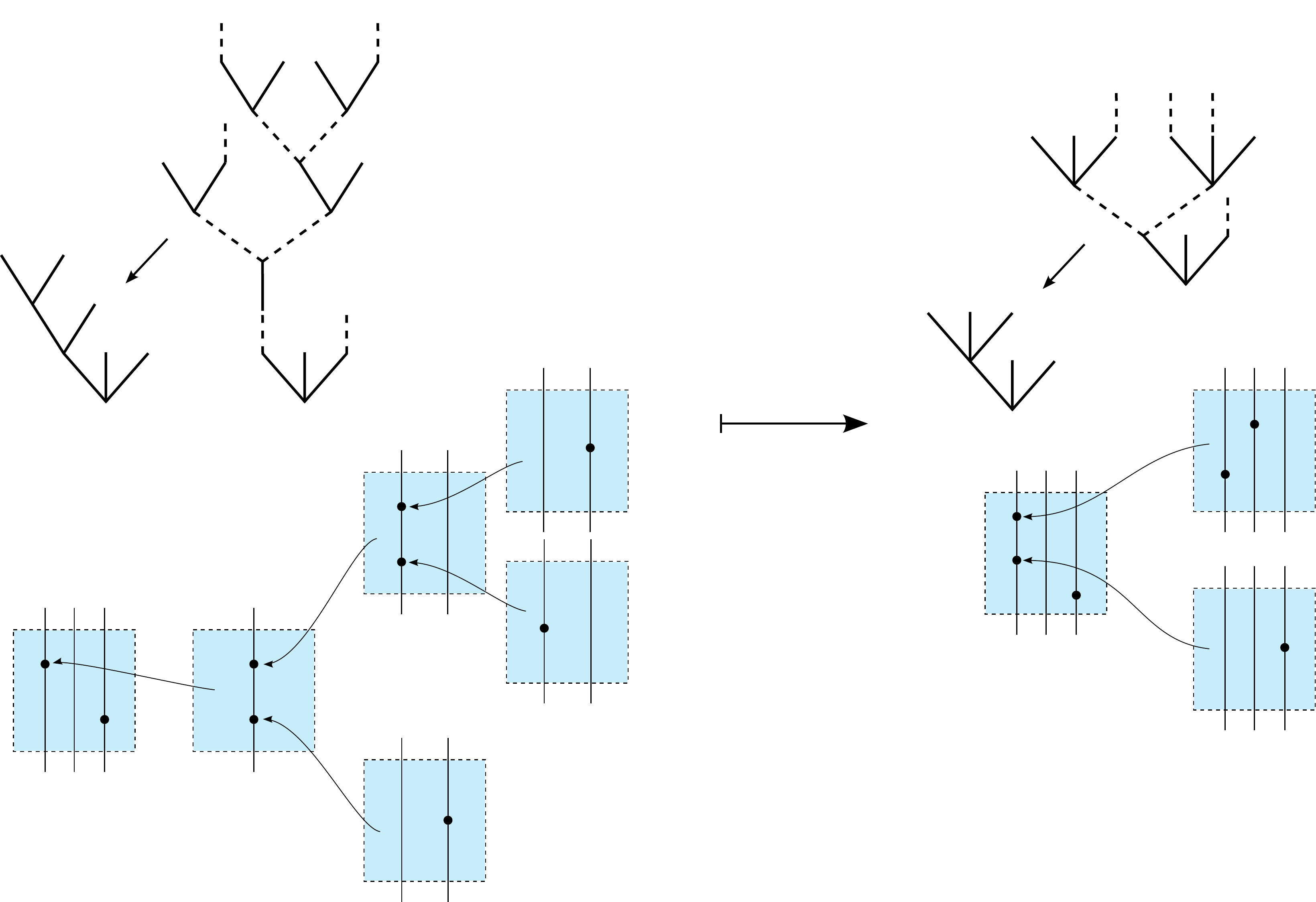
\end{figure}
\noindent
On the left is a tree-pair $2T \in W_{11101}^\bC$, together with numbers labeling the elements of $V_\comp(T_b) \setminus\{\alpha_\root\}$ and $V_\inte(T_s)\setminus\{\rho_\root\}$ that together comprise an element of $2p_{2T}$.
Below $2T$, we have depicted a typical element of $\ol{2M}_{11101}^\bC$ with combinatorial type $2T$.
On the right is the image of the indicated element of $2p_{2T}$ under $2g_{2T}$.
\null\hfill$\triangle$
\end{example}

\begin{lemma}
For any $2T \in W_\bn^\bC$, $2g_{2T}$ is an inclusion of posets, with image $2g_{2T}(2p_{2T}) = [2T,2T_\bn^\top]$.
\end{lemma}

\begin{proof}
We will prove the containment $[2T,2T_\bn^\top] \subset 2g_{2T}(2p_{2T})$; the opposite containment follows from a similar argument.
By induction, it suffices to show that for any $2T' > 2T$ with $d(2T') = d(2T) + 1$, $2T'$ is in the image of $2g_{2T}$.
As in \cite{b:2ass}, $2T$ is the result of performing a type-1, type-2, or type-3 move on $2T'$, so we prove the inductive hypothesis in these three cases.
\begin{itemize}
\item
Suppose that $2T$ is the result of performing a type-1 move on $2T'$, by creating a new component vertex $\alpha_0$ in $T_b$ via a manipulation of the following form:
\begin{figure}[H]
\centering
\def\svgwidth{0.45\columnwidth}
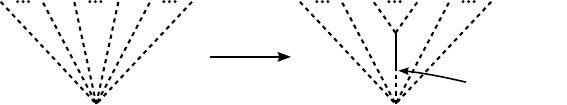
\end{figure}
\noindent
Set $q_\alpha$ to be 1 if $\alpha = \alpha_0$ and 0 otherwise, and $p_\rho$ to be 0.
Then $\bigl((q_\alpha),(p_\rho)\bigr)$ clearly satisfies \eqref{eq:Wn_model_coherences}, and $2T' = 2g_{2T}\bigl((q_\alpha),(p_\rho)\bigr)$.

\medskip

\item
Suppose that $2T$ is the result of performing a type-2 move on $2T'$.
This means that we create a new interior vertex $\sigma_0$ in $T_s$ via a manipulation of the following form:
\begin{figure}[H]
\centering
\def\svgwidth{0.4\columnwidth}
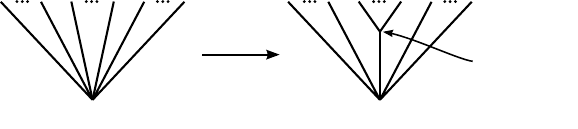
\end{figure}
\noindent
In addition, for every $\alpha_0 \in V_\comp^{\geq2}(T_b)$ lying over $\rho_0$, we create new component vertices $\alpha_1,\ldots,\alpha_k$ via a manipulation like this:
\begin{figure}[H]
\centering
\def\svgwidth{0.5\columnwidth}
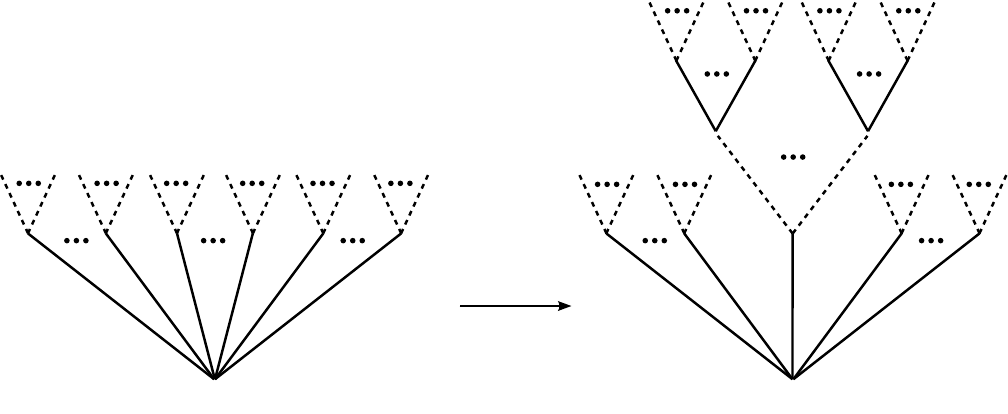
\end{figure}
\noindent
Define $\bigl((q_\alpha),(p_\rho)\bigr)$ by setting $q_\alpha$ to be 1 if $\alpha \in \{\alpha_1,\ldots,\alpha_k\}$ and 0 otherwise, and by setting $p_\rho$ to be 1 if $\rho = \sigma_0$ and 0 otherwise.
Then $\bigl((q_\alpha),(p_\rho)\bigr)$ satisfies \eqref{eq:Wn_model_coherences} --- e.g.\ both sides of the second equation in \eqref{eq:Wn_model_coherences} are 1 if $\rho = \sigma_0$, and both sides are 0 otherwise --- and $2T' = 2g_{2T}\bigl((q_\alpha),(p_\rho)\bigr)$.

\medskip

\item
Suppose that $2T$ is the result of performing a type-3 move on $2T'$, by creating new component vertices $\beta_1,\ldots,\beta_k$ via a manipulation of the following form:
\begin{figure}[H]
\centering
\def\svgwidth{0.4\columnwidth}
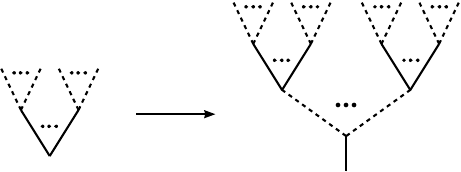
\end{figure}
\end{itemize}
Set $q_\alpha$ to be 1 if $\alpha \in \{\alpha_1,\ldots,\alpha_k\}$ and 0 otherwise, and $p_\rho$ to be 0.
Then $\bigl((q_\alpha),(p_\rho)\bigr)$ satisfies \eqref{eq:Wn_model_coherences} --- e.g.\ in the second line of those coherences, both sides are zero --- and $2T = 2g_{2T}\bigl((q_\alpha),(p_\rho)\bigr)$.
\end{proof}

\subsection{Construction of $\ol{2M}_\bn^\bC$ as a set, and of an atlas}
\label{ss:2Mn-bar_charts}

\begin{definition}
\label{def:SWC}
A \emph{stable plane-tree of type $\bn \in \bZ_{\geq0}^r\setminus\{\bzero\}$} is a triple
\begin{align}
\Bigl(2T=(T_b \sr{\pi}{\to} T_s), (\bx_\rho)_{\rho \in V_\inte(T_s)}, (\bz_\alpha)_{\alpha \in V_\comp(T_b)}\Bigr),
\end{align}
where:
\begin{itemize}
\item $2T$ is an element.

\item For $\rho \in V_\inte(T)$, $\bx_\rho$ is a tuple in $\bC^{\#\incom(\rho)} \setminus \Delta$, where $\Delta$ is the fat diagonal.

\item For $\alpha \in V_\comp(T_b)$, $\bz_\alpha \subset \bC^2 \setminus \Delta$ is a collection
\begin{gather}
\bz_\alpha = \left(z_{\alpha,ij} = (x_{\alpha,i},y_{\alpha,ij}) \:\left|\: {{\incom(\alpha) \eqqcolon (\beta_1,\ldots,\beta_{\#\incom(\alpha)}),}
\atop
{1\leq i\leq \#\incom(\alpha), \: 1 \leq j \leq \#\incom(\beta_i)}},
\eqref{eq:plane-tree_coherences}\right.\right),
\end{gather}
where \eqref{eq:plane-tree_coherences} is the following collection of coherences:
\begin{gather}
(x_{\alpha,1},\ldots,x_{\alpha,\#\incom(\alpha)}) = (x_{\pi(\alpha),1},\ldots,x_{\pi(\alpha),\#\incom(\pi(\alpha))})
\qquad
\forall \alpha \in V_\comp^{\geq2}(T_b).
\label{eq:plane-tree_coherences}
\end{gather}
\end{itemize}
We say that two stable plane-trees $\bigl(2T,(\bx_\rho),(\bz_\alpha)\bigr)$, $\bigl(2T',(\bx'_\rho),(\bz'_\alpha)\bigr)$ are \emph{isomorphic} if there is an isomorphism of stable tree-pairs $2f\colon 2T \to 2T'$ and functions $V_\inte(T_s) \to G_1\colon \rho \mapsto \phi_\rho$ and $V_\comp(T_b) \to G_2\colon \alpha \mapsto \psi_\alpha$ such that:
\begin{gather}
z'_{f_b(\alpha),ij} = \psi_\alpha(z_{\alpha,ij}) \:\:\forall\:\: \alpha \in V_\comp(T_b),
\qquad
x'_{f_s(\rho),i} = \phi_\rho(x_{\rho,i}) \:\:\forall\:\: \rho \in V_\inte(T_s), \\
p(\psi_\alpha) = \phi_{\pi(\alpha)} \:\:\forall\:\: \alpha \in V_\comp^{\geq 2}(T_b). \nonumber
\end{gather}

\noindent We denote the collection of stable plane-trees of type $\bn$ by $\SPT_\bn$, and we define the \emph{moduli space $\ol{2M}_\bn$ of stable plane-trees of type $\bn$} to be the set of isomorphism classes of stable plane-trees of this type.
For any stable plane-tree $2T$ of type $\bn$, define the corresponding \emph{strata} $\SPT_{\bn,2T} \subset \SPT_\bn$, $\ol{2M}_{\bn,2T} \subset \ol{2M}_\bn$ to be the set of all stable plane-trees (resp.\ isomorphism classes thereof) of the form $\bigl(2T,(\bx_\rho),(\bz_\alpha)\bigr)$.
We say that a stable plane-tree is \emph{smooth} if its underlying stable plane-tree $2T$ has the property that $V_\inte(T_s)$ and $V_\comp(T_b)$ each contain only one element; we denote a smooth stable plane-tree by the pair $(\bx,\bz) \in \bC^r\times\bC^{|\bn|+r}$ associated to the roots of $T_s$ resp.\ $T_b$.
\null\hfill$\triangle$
\end{definition}

Next, we will equip $\ol{2M}_\bn^\bC$ with the structure of an algebraic variety.
We will do so by constructing an atlas for $\ol{2M}_\bn^\bC$, in which the transition maps are algebraic morphisms.
Specifically, we will associate to each tree-pair $2T \in W_\bn^\bC$ with $d(2T)=0$ a chart $\varphi_{2T}\colon 2X_{2T} \to \ol{2M}_\bn^\bC$.
\footnote{
We could, as in \S\ref{s:Mr-bar}, construct an atlas for every $2T \in W_\bn^\bC$.
We impose the restriction $d(2T) = 0$ because cuts down on the proliferation of notation.
}
Our construction of $\varphi_{2T}$ and our verification that these charts satisfy the necessary properties will be straightforward, but rather involved.

\begin{definition}
If $2T \in W_\bn^\bC$ is a stable tree-pair, then a \emph{slice} $2S$ of $2T$ consists of the following data:
\begin{itemize}
\item
A slice $\bs$ of the seam tree $T_s \in K_r^\bC$.

\medskip

\item
A triple $(S_0, S_1, S_{(0,0)})$ of functions
\begin{align}
S_0, S_1
\colon
V_\comp^{\geq2}(T_b)
\to
V_\seam(T_b),
\qquad
S_{(0,0)}
\colon
V_\comp^{\geq2}(T_b)
\to
V_\comp(T_b) \cup V_\mk(T_b).
\end{align}

\medskip

\item
A pair $(S_{(0,0)}, S_{(0,1)})$ of functions
\begin{align}
S_{(0,0)}, S_{(0,1)}
\colon
V_\comp^1(T_b)
\to
V_\comp(T_b) \cup V_\mk(T_b).
\end{align}
\end{itemize}
We require that these data satisfy the following properties:
\begin{itemize}
\item
For any $\alpha \in V_\comp^{\geq2}(T_b)$, the following equalities hold:
\begin{align}
s_i(\pi(\alpha))
=
\pi(S_i(\alpha)),
\quad
i \in \{0,1\}.
\end{align}

\medskip

\item
For any $\alpha \in V_\comp^{\geq2}(T_b)$, $S_0(\alpha)$ and $S_1(\alpha)$ are distinct elements of $\incom(\alpha)$, and $S_{(0,0)}(\alpha)$ is an incoming vertex of $S_0(\alpha)$.
For any $\alpha \in V_\comp^1(T_b)$, $S_{(0,0)}(\alpha)$ and $S_{(0,1)}(\alpha)$ are distinct elements of $\incom(S_0(\alpha))$.
\end{itemize}
For a sliced tree $2T = (2T, 2S)$, we define a map
\begin{align}
&2C_{2T}
\colon
\prod_{\rho \in V_\inte(T_s)}
\bigl((\bC\setminus\{0,1\})^{\#\incom(\rho)-2}\setminus\Delta\bigr)
\times
\prod_{{\alpha \in V_\comp^1(T_b),}
\atop
{\incom(\alpha)\eqqcolon\{\beta\}}}
\bigl(\bC^{\#\incom(\beta)}\setminus\Delta\bigr)
\times
\\
&\hspace{0.5in}
\times
\prod_{\alpha \in V_\comp^{\geq2}(T_b)}
\biggl(
\bigl((\bC\setminus\{0\})^{\#\incom(S_0(\alpha))-1}\setminus\Delta\bigr)
\times
\prod_{\beta \in \incom(\alpha)\setminus\{S_0(\alpha)\}}
\bigl(\bC^{\#\incom(\beta)}\setminus\Delta\bigr)
\biggr)
\to
\ol{2M}_\bn^\bC.
\nonumber
\end{align}
This map is defined by analogy with the map $C_T$ from Definition~\ref{def:slice_and_CT}.
In this case, for $\alpha \in V_\comp^{\geq2}(T_b)$, $S_0(\alpha)$ and $S_1(\alpha)$ correspond to the lines that are fixed at positions 0 and 1, and $S_{(0,0)}(\alpha)$ is the point on the line corresponding to $S_0(\alpha)$ that is fixed at position $(0,0)$; for $\alpha \in V_\comp^1(T_b)$, $S_{(0,0)}(\alpha)$ and $S_{(0,1)}(\alpha)$ are the points on the (only) line which are fixed at positions $(0,0)$ and $(0,1)$.
Then $2C_{2T}$ restricts to a bijection from its domain to $\ol{2M}_{\bn,2T}^\bC$.
\null\hfill$\triangle$
\end{definition}

For any sliced $2T \in W_\bn^\bC$ with $d(2T) = 0$, we aim to define its associated local model $2X_{2T}$.
We begin by defining $\wt{2X}_{2T}$, which will contain $2X_{2T}$ as a Zariski-open subset:
\begin{align}
\label{eq:2X_2T-tilde}
&\wt{2X}_{2T}
\coloneqq
\left\{\left.\bigl(
(b_\rho)_{\rho \in V_\inte(T_s)\setminus\{\rho_\root\}},
(a_\alpha)_{\alpha \in V_\comp(T_b)\setminus\{\alpha_\root\}}
\bigr)
\in
\bC^{\#V_\comp(T_b) + \#V_\inte(T_s) - 2}
\:\right|\:
\eqref{eq:2X_2T-tilde_coherences}\right\}.
\end{align}
Here \eqref{eq:2X_2T-tilde_coherences} is the following collection of coherences:
\begin{gather}
\prod_{\gamma \in [\alpha_1,\beta)} a_\gamma
=
\prod_{\gamma \in [\alpha_2,\beta)} a_\gamma
\qquad
\forall \: \alpha_1,\alpha_2 \in V_\comp^{\geq2}(T_b),
\beta \in V_\comp^1(T_b)
:
\pi(\alpha_1) = \pi(\alpha_2) = \pi(\beta),
\label{eq:2X_2T-tilde_coherences}
\\
b_\rho
=
\prod_{\gamma \in [\alpha,\beta_\alpha)} a_\gamma
\qquad
\forall \rho \in V_\inte(T_s)\setminus\{\rho_\root\},
\alpha \in V_\comp^{\geq2}(T_b)\cap\pi^{-1}\{\rho\},
\nonumber
\end{gather}
where in the second inequality we define $\beta_\alpha$ to be the first element of $V_\comp^{\geq2}(T_b)$ that the path from $\alpha$ to $\alpha_\root$ passes through.
\null\hfill$\triangle$

\begin{example}
\label{ex:2Mn_model}
Consider the tree-pairs $2T^1, 2T^2$ depicted in the following figure:
\begin{figure}[H]
\centering
\def\svgwidth{0.8\columnwidth}
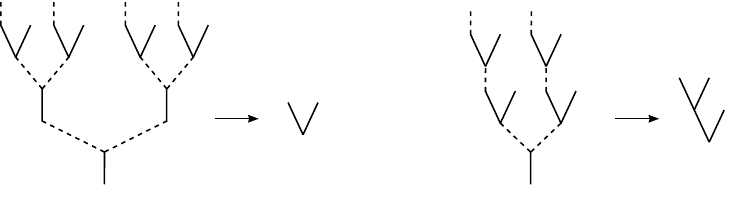
\end{figure}
\noindent
We have labeled the interior dashed edges of $T_b$ and the interior edges of the seam trees with their associated variables.
\begin{itemize}
\item
The local model $\wt{2X}_{2T}$ associated to $2T^{(1)}$ is
\begin{align}
\left\{(a,b,c,d,e,f) \in \bC^6
\:\left|\:
{
{c=d,e=f}
\atop
{ac=ad=be=bf}
}
\right.\right\},
\end{align}
which we can identify with the quadric cone
\begin{align}
\bigl\{(a,b,c,e) \in \bC^4
\:|\:
ac=be\bigr\}.
\end{align}

\item
The local model associated to $2T^{(2)}$ is
\begin{align}
\bigl\{(a,b,c,d,e) \in \bC^5
\:|\:
a=b,c=e=d
\bigr\},
\end{align}
which we can identify with $\bC^2 \ni (a,e)$.
\null\hfill$\triangle$
\end{itemize}
\end{example}

\begin{definition}
\label{def:2Mn_local_model}
Fix $2C = \bigl(2T,\bigl((\bx_\rho),(\bz_\alpha)\bigr)\bigr)$ a stable plane tree, and distinct vertices $\alpha \in V_\comp(T_b)$, $\beta \in (T_b)_\alpha \cap (V_\comp(T_b) \cup V_\mk(T_b))$.
We define a polynomial $2p_{\alpha\beta}^{2C}$ in variables $a_\delta$, $\delta \in V_\comp(T_b) \setminus \{a_\root\}$ like so:
\begin{align}
2p_{\alpha\beta}^{2C}(\ba)
\coloneqq
\sum_{\gamma \in [\alpha,\beta[}
\Bigl(z_{\gamma\beta}
\prod_{\delta \in ]\alpha,\gamma]} a_\delta\Bigr).
\end{align}
This allows us to define, for any sliced $2T \in W_\bn^\bC$ with $d(2T) = 0$, the associated local model $2X_{2T}$:
\begin{align}
\label{eq:Wn_local_model_def}
2X_{2T}
\coloneqq
\wt{2X}_{2T}
\setminus
\Bigl(
\bigcup_{i < j}
Z(q_{ij})
\cup
\bigcup_k
\bigcup_{i < j}
Z(Q_{ij}^k)
\Bigr)
\end{align}
Here we have defined $q_{ij}$ as in \S\ref{s:Mr-bar}, and, for any $i, j, k$, we have defined $Q_{ij}^k$ to be the largest polynomial factor of $p_{\rho_\root\mu_{ki}}^{2C_{2T}()} - p_{\rho_\root\mu_{ki}}^{2C_{2T}()}$ not divisible by a monomial in $\ba$ and $\bb$.
The local model $2X_{2T}$ is stratified as $2X_{2T} = \bigcup_{2T'\geq 2T} 2X_{2T,2T'}$.
\null\hfill$\triangle$
\end{definition}

Analogously to Def.~\ref{def:Mr_charts}, we can now define the charts for $\ol{2M}_\bn^\bC$.

\begin{definition}
\label{def:2Mn_chart}
Fix a sliced tree-pair $2T \in W_\bn^\bC$ with $d(2T) = 0$, equipped with a slice $2S$.
We define $2\varphi_{2T}\colon 2X_{2T} \to \ol{2M}_\bn^\bC$ like so:
\begin{gather}
2\varphi_{2T}(
\bb,
\ba
)
\coloneqq
\Bigl(
2g_{2T}\bigl(2\pi_{2T}(\bb,\ba)\bigr),
\bigl(
\bigl(\bx_\rho^{\varphi_T(\bb)}\bigr),
\bigl(\bz_\alpha^{2\varphi_{2T}(\bb,\ba)}\Bigr)
\bigr)
\bigr),
\quad
\bz_\alpha^{2\varphi_{2T}(\bb,\ba)}
\coloneqq
\Bigl(2p_{\alpha\beta}^{2C_{2T}()}(\ba)\Bigr).
\end{gather}
$\null\hfill\triangle$
\end{definition}

\noindent
In Lemma~\ref{lem:X_T_strats}, we showed that $X_T$ is exactly defined so that, for every $\bigl((\bx_\rho),\bb\bigr)$, the glued map $\varphi_T\bigl((\bx_\rho),\bb\bigr)$ has the property that on each copy of $\bC$, no two special points coincide.
An exactly analogous result is also true for $2X_{2T}$, and we omit the proof.

\begin{example}
In this example, we will work out two of the charts on $\ol{2M}_{21}^\bC$ and the transition map between them.
Specifically, we will consider the charts $\phi, \psi$ associated to the following two vertices of $\ol{2M}_{21}^\bC$:
\begin{figure}[H]
\centering
\def\svgwidth{1.0\columnwidth}
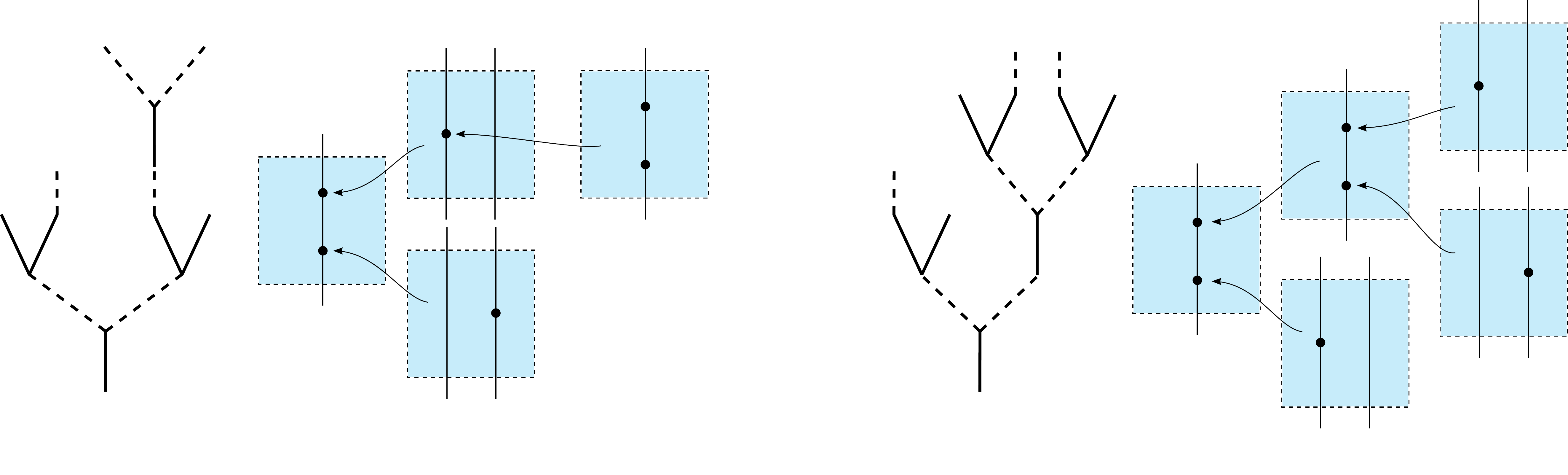
\end{figure}
\noindent
The local models are $X_\phi = \bC^2 \ni (a,b)$ and $X_\psi = \{(a,b,c) \in \bC^3 \:|\: a = bc\} \setminus \{(0,0,-1)\}$.

First, we depict the image of $(a,b) \in X_\phi$ under $\phi$:
\begin{figure}[H]
\centering
\def\svgwidth{1.0\columnwidth}
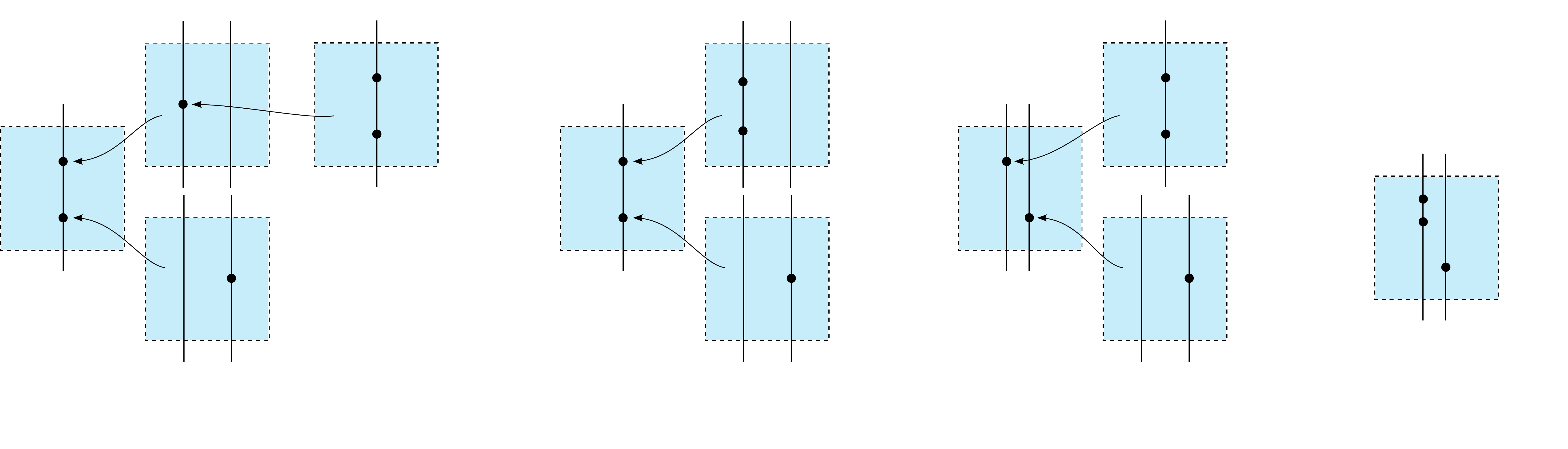
\end{figure}

\noindent
Next, we depict the image of $(a,b,c) \in X_\psi$ under $\psi$:
\begin{figure}[H]
\centering
\def\svgwidth{1.0\columnwidth}
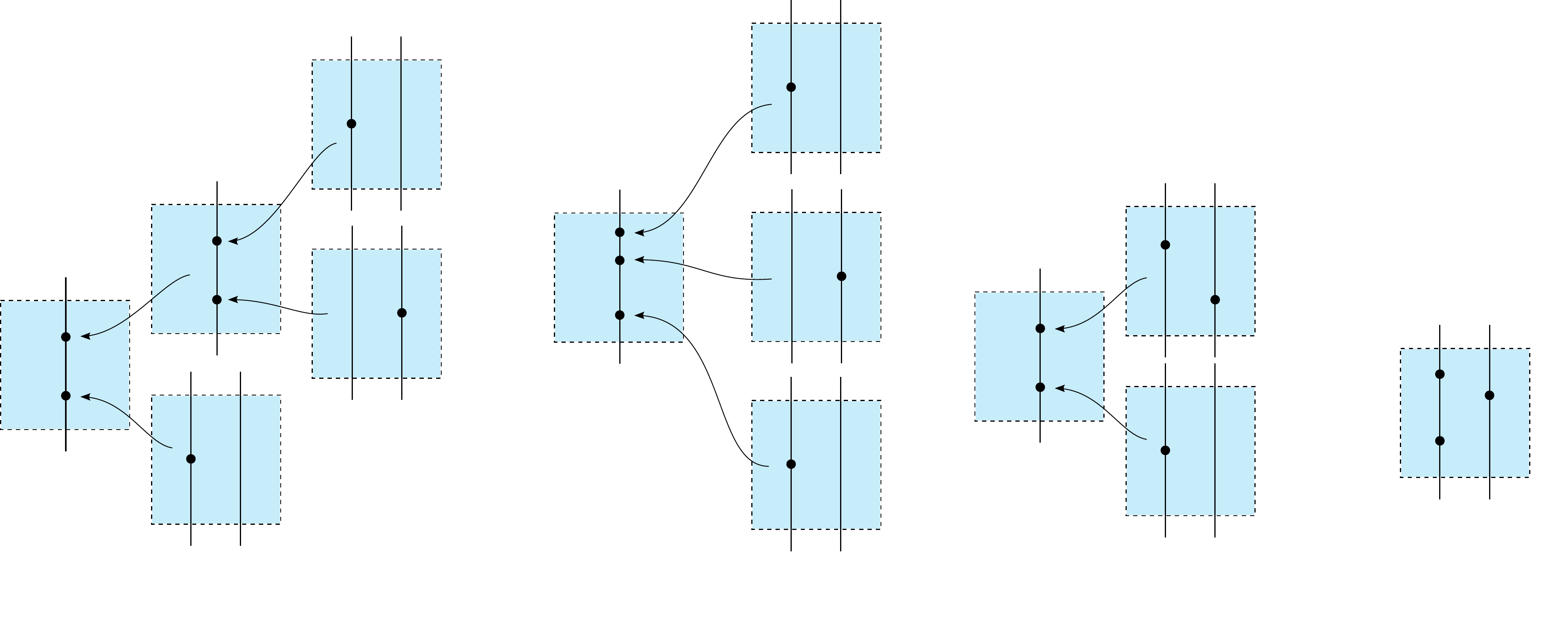
\end{figure}

Finally, we will compute $\psi^{-1}\circ\phi$.
The domain $\phi^{-1}\bigl(\phi(X_\phi) \cap \psi(X_\psi)\bigr)$ of this transition map is $(\bC\setminus\{0\})^2 \setminus \{ab = -1\}$, and in the following figure we compute it:
\begin{figure}[H]
\centering
\def\svgwidth{1.0\columnwidth}
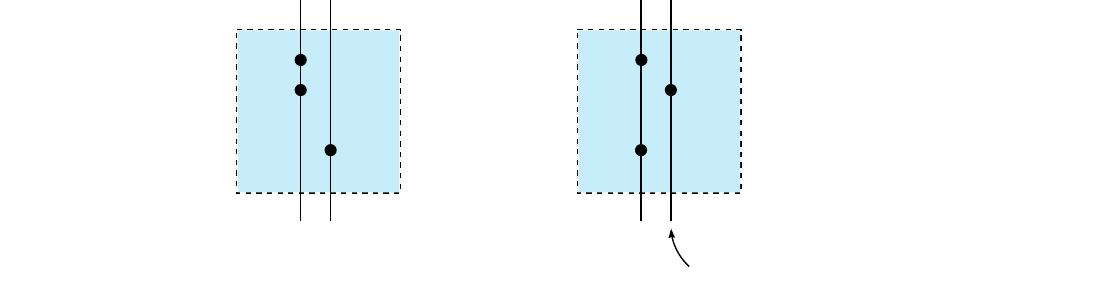
\null\hfill$\triangle$
\end{figure}
\end{example}

The following two lemmas can be proven in a fashion analogous to Lemmas~\ref{lem:image_of_Kr_chart} and \ref{lem:Mr_transitions_algebraic}.

\begin{lemma}
\label{lem:image_of_Wn_chart}
For any sliced stable tree-pair $2T$ with $d(2T) = 0$, $2\varphi_{2T}$ is injective, with image
\begin{align}
2\varphi_{2T}(2X_{2T}) = \bigcup_{2T' \geq 2T} \ol{2M}_{\bn,2T'}.
\end{align}
\end{lemma}

\begin{lemma}
\label{lem:2Mn_transitions_algebraic}
For any sliced tree-pairs $(2T_1,2\bs^1), (2T_2,2\bs^2) \in \ol{2M}_\bn^\bC$ with $d(2T_1) = d(2T_2) = 0$, the transition map $2\varphi_{2T_2}^{-1} \circ 2\varphi_{2T_1}\colon 2X_{2T_1,2T_\bn^\top} \to 2X_{2T_2,2T_\bn^\top}$ is a morphism.
\end{lemma}

\begin{lemma}
\label{lem:2Mn_variety}
There is a unique topology on $\ol{2M}_\bn^\bC$ such that $\bigl(\ol{2M}_\bn^\bC, (2\varphi_{2T})\bigr)$ is a variety, where we allow $2T$ to range over all tree-pairs $2T \in W_\bn^bC$ equipped with a slice and with $d(2T) = 0$.
\end{lemma}

\begin{proof}
The proof of this lemma is very similar to the proof of Lemma~\ref{lem:Mr_variety}, using Lemmas \ref{lem:image_of_Wn_chart} and \ref{lem:2Mn_transitions_algebraic} in place of Lemmas \ref{lem:image_of_Kr_chart} and \ref{lem:Mr_transitions_algebraic}.
\end{proof}

Finally, a proof similar to that of Lemma~\ref{lem:Mr_sep_proper} establishes the following lemma.

\begin{lemma}
\label{lem:2Mn_sep_proper}
$\ol{2M}_\bn^\bC$ is separated and proper.
\end{lemma}

\subsection{The local models $2X_{2T}$ are reduced and normal}
\label{ss:2Mn-bar_reduced_normal}

In this subsection, we will show that $\ol{2M}_\bn^\bC$ is locally toric intersection.
We will do so by showing that $2X_{2T}$ is an affine toric variety, for $2T$ with $d(2T) = 0$.
(By Lemma~\ref{lem:image_of_Wn_chart}, it suffices to consider only these $2T$'s.)
The following is our definition of a toric variety, which includes the requirement of normality:

\begin{definition}
A \emph{toric variety} over a field $k$ is a normal algebraic variety $X$ over $k$ that contains as an open dense subset an algebraic torus $T \subset X$, and that has an action by $T$ that extends the action of $T$ on itself.
\null\hfill$\triangle$
\end{definition}

A key to our proof that $2X_{2T}$ is toric is the following lemma, which appears in \cite{bier_mil:toric_binomial} but goes back at least to \cite{es:binomial}:

\medskip

\noindent
{\bf Lemma 2.2, \cite{bier_mil:toric_binomial}.}
{\it An ideal in $k[\bx]$ is toric if and only if it is integrally closed, prime, and generated by binomials of the form $\bx^a - \bx^b$.}

\medskip

The following is the main result of this subsection.

\begin{lemma}
\label{lem:2X_2T_reduced_normal}
For $2T \in W_\bn^\bC$ with $d(2T) = 0$, the local model $2X_{2T}$ is reduced and normal.
\end{lemma}

\noindent Our proof of this lemma is rather technical.
We suggest that the reader refer as needed to Example~\ref{ex:reduced_and_normal}, in which we demonstrate the proof of Lemma~\ref{lem:2X_2T_reduced_normal} in an example.

Before we prove Lemma~\ref{lem:2X_2T_reduced_normal}, we develop an equivalent presentation of $\wt{2X}_{2T}$.
We begin by phrasing $\wt{2X}_{2T}$ as $\Spec$ of a monoid algebra:
\begin{gather}
\label{eq:2X_2T_as_monoid_algebra}
\wt{2X}_{2T}
=
\Spec\: \bC\Bigl[ \bZ_{\geq0}^N / L_{2T} \Bigr],
\qquad
N
\coloneqq
\#\bigl(V_\comp(T_b)\setminus\{\alpha_\root\}\bigr)
+
\#\bigl(V_\inte(T_s)\setminus\{\rho_\root\}\bigr),
\\
L_{2T}
\coloneqq
\bZ\langle S_{2T}\rangle,
\quad
S_{2T}
\coloneqq
\Bigl\{
\sum_{\gamma \in [\alpha,\beta)}
a_\gamma
-
\sum_{\gamma \in [\alpha',\beta)}
a_\gamma
\Bigr\}_{
{\alpha,\alpha' \in V_\comp^{\geq2}(T_b),}
\atop
{\pi(\alpha)=\pi(\alpha'),
\beta < \alpha,\alpha'}
}
\cup
\Bigl\{
b_\rho
-
\sum_{\gamma \in [\alpha,\beta_\alpha)}
a_\gamma
\Bigr\}_{
{\alpha \in V_\comp^{\geq2}(T_b),}
\atop
{\pi(\alpha) = \rho \neq \rho_\root}
}.
\nonumber\end{gather}
By $\bZ^N_{\geq0}/L_{2T}$, we mean the monoid defined as follows.
First, consider the quotient group $\bZ^N / L_{2T}$.
Then, consider the projection $\pi\colon \bZ^N \to \bZ^N / L_{2T}$, and set $\bZ^N_{\geq0}/L_{2T} \coloneqq \pi(\bZ^N)$.
Note that $\bZ^N_{\geq0} / L_{2T}$ is a submonoid of $\bZ^N / L_{2T}$, and moreover that $\bZ^N / L_{2T}$ is the groupification of $\bZ^N_{\geq0} / L_{2T}$.

We now define a collection of elements of $\bZ^N$, which we call the \emph{canonical generators of $L_{2T}$}.
Immediately after this definition, we will prove a lemma that justifies this terminology.

\begin{definition}
\label{def:can_gens}
Fix a planar embedding of $T_b$.
For any $\alpha \in V_\comp^{\geq2}(T_b)$ (except for the one indicated in the third bullet), define an element of $\bZ^N$ in the following way:
\begin{itemize}
\item
Suppose that $\alpha$ is not the last element of $\pi^{-1}(\pi(\alpha)) \cap V_\comp^{\geq2}(T_b)$ with respect to inorder traversal.
Define $\alpha'$ to be the next element of $\pi^{-1}(\pi(\alpha)) \cap V_\comp^{\geq2}(T_b)$, and define $\beta,\beta' \in V_\comp(T_b)$ like so:
\begin{itemize}
\item
Set $\gamma$ (resp.\ $\gamma'$) to be the first place where the path from $\alpha$ to $\alpha_\root$ (resp.\ the path from $\alpha'$ to $\alpha_\root$) meets $V_\comp(T_b)$.
If $\gamma \neq \gamma'$, set $\beta \coloneqq \gamma$ and $\beta' \coloneqq \gamma'$.

\item
If, in the previous bullet, we have $\gamma = \gamma'$, then set $\beta$ and $\beta'$ both equal to the first element of $V_\comp(T_b)$ where the path from $\alpha$ to $\alpha_\root$ and the path from $\alpha'$ to $\alpha_\root$ intersect.
\end{itemize}
We associate to $\alpha$ the following vector:
\begin{align}
\label{eq:can_gen_1}
\bv_\alpha
\coloneqq
\sum_{\gamma \in [\alpha,\beta)} a_\gamma
-
\sum_{\gamma \in [\alpha',\beta')} a_\gamma.
\end{align}

\medskip

\item
Set $\rho \coloneqq \pi(\alpha)$, and suppose $\rho \neq \rho_\root$ and that $\alpha$ is the last element of $\pi^{-1}(\rho) \cap V_\comp^{\geq2}(T_b)$ with respect to inorder traversal.
Then we associate to $\alpha$ the vector
\begin{align}
\label{eq:can_gen_2}
\bv_\alpha
\coloneqq
b_\rho
-
\sum_{\gamma \in [\alpha,\beta_\alpha)} a_\gamma,
\end{align}
where, as before, $\beta_\alpha$ denotes the first element of $V_\comp^{\geq2}(T_b)$ through which the path from $\alpha$ to $\alpha_\root$ passes.

\medskip

\item
If $\alpha$ is the last element of $\pi^{-1}(\rho_\root) \cap V_\comp^{\geq2}(T_b)$ with respect to inorder traversal, then we associate no element of $\bZ^N$ to $\alpha$.
\end{itemize}
Then the collection of $\bv_\alpha$'s defined in these bullets is a generating set for $L_{2T}$, and we refer to the $\bv_\alpha$'s as the \emph{canonical generators} associated to $2T$ and the chosen planar embedding of $T_b$.
\null\hfill$\triangle$
\end{definition}

\begin{lemma}
\label{lem:can_gens}
The canonical generators generate $L_{2T}$.
\end{lemma}

\begin{proof}
{\bf Step 1:}
For any $\alpha \in V_\comp^{\geq2}(T_b) \setminus \{\alpha_0\}$, $\bv_\alpha$ lies in $L_{2T}$.

\medskip

\noindent
First, suppose that $\alpha$ is not the last element of $\pi^{-1}(\pi(\alpha)) \cap V_\comp^{\geq2}(T_b)$ with respect to inorder traversal.
Defining $\alpha'$, $\beta$, and $\beta'$ as in the first bullet of Definition~\ref{def:can_gens}, we must show that $\bv_\alpha = \sum_{\gamma \in [\alpha,\beta)} a_\gamma - \sum_{\gamma \in [\alpha',\beta')} a_\gamma$ can be written as a sum of elements of $S_{2T}$.
If $\beta = \beta'$, then $\bv_\alpha$ is simply an element of $S_{2T}$.
Otherwise, set $\beta''$ to be the first element of $V_\comp(T_b)$ where the paths from $\beta$ to $\alpha_\root$ resp.\ from $\beta'$ to $\alpha_\root$ intersect, and decompose $\bv_\alpha$ like so:
\begin{align}
\sum_{\gamma \in [\alpha,\beta)} a_\gamma - \sum_{\gamma \in [\alpha',\beta')} a_\gamma
=
\biggl(\sum_{\gamma \in [\alpha,\beta'')} a_\gamma - \sum_{\gamma \in [\alpha',\beta'')} a_\gamma\biggr)
-
\biggl(\sum_{\gamma \in [\beta,\beta'')} a_\gamma - \sum_{\gamma \in [\beta',\beta'')} a_\gamma\biggr).
\end{align}

Second, set $\rho \coloneqq \pi(\alpha)$, and suppose $\rho \neq \rho_\root$ and that $\alpha$ is the last element of $\pi^{-1}\{\rho\} \cap V_\comp^{\geq2}(T_b)$ with respect to inorder traversal.
Then $\bv_\alpha = b_\rho - \sum_{\gamma \in [\alpha,\beta_\alpha)} a_\gamma$ is an element of $S_{2T}$.

\medskip

\noindent
{\bf Step 2:}
Fix any $\alpha,\alpha' \in V_\comp^{\geq2}(T_b)$ and $\beta \in V_\comp(T_b)$ with $\pi(\alpha) = \pi(\alpha')$ and $\beta < \alpha, \alpha'$.
Then $\bv_\alpha$ can be written as a sum of canonical generators of the form \eqref{eq:can_gen_1}.

\medskip
 
\noindent
Without loss of generality, we may assume that $\alpha,\alpha'$ are consecutive elements of $\pi^{-1}(\pi(\alpha)) \cap V_\comp^{\geq2}(T_b)$ with respect to inorder traversal.
Indeed, by switching $\alpha$ and $\alpha'$ if necessary, we may assume that $\alpha'$ follows $\alpha$, and if we denote by $(\alpha = \alpha_1, \alpha_2,\ldots,\alpha_k = \alpha')$ the interval from $\alpha$ to $\alpha'$ in $\pi^{-1}(\pi(\alpha)) \cap V_\comp^{\geq2}(T_b)$ with respect to inorder traversal, we can write
\begin{align}
\sum_{\gamma \in [\alpha,\beta)} a_\gamma - \sum_{\gamma \in [\alpha',\beta)} a_\gamma
\quad
=
\quad
\sum_{i=1}^{k-1} \: \Bigl( \sum_{\gamma \in [\alpha_i,\beta)} a_\gamma - \sum_{\gamma \in [\alpha_{i+1},\beta)} a_\gamma \Bigr).
\end{align}
Note that for each $i$, $\beta$ is on the path from $\alpha_i$ to $\alpha_\root$.
We may also assume that $\beta$ is the first element of $V_\comp(T_b)$ where the paths from $\alpha$ to $\alpha_\root$ resp.\ from $\alpha'$ to $\alpha_\root$ intersect.

If the path from $\alpha$ to $\beta$ does not pass through any elements of $V_\comp(T_b)$ besides $\alpha$ and possibly $\beta$, then $\bv_\alpha$ is already a canonical generator.
Otherwise, define $\delta$ to be the first element of $V_\comp^{\geq2}(T_b)$ other than $\alpha$ through which the path from $\alpha$ to $\alpha_\root$ passes; define $\delta'$ similarly, using the path from $\alpha'$ to $\alpha_\root$.
We then have
\begin{align}
\bv_\alpha
=
\biggl(\sum_{\gamma \in [\alpha,\delta)} a_\gamma - \sum_{\gamma \in [\alpha',\delta')} a_\gamma\biggr)
-
\biggl(\sum_{\gamma \in [\delta,\beta)} a_\gamma - \sum_{\gamma \in [\delta',\beta)} a_\gamma\biggr).
\end{align}
Moreover, the first parenthesized expression in this difference is the canonical generator $\bv_\alpha$, so we have reduced the task at hand to showing that the second parenthesized expression can be decomposed as a sum of canonical generators.
Proceeding inductively in this fashion, we obtain the desired conclusion.

\medskip

\noindent
{\bf Step 3:}
Fix any $\alpha \in V_\comp^{\geq2}(T_b)$, set $\rho\coloneqq \pi(\alpha)$, and suppose $\rho \neq \rho_\root$.
Then $\bv_\alpha$ can be written as a sum of canonical generators of the form \eqref{eq:can_gen_1} and \eqref{eq:can_gen_2}.

\medskip

\noindent
If $\alpha$ is the last element of $\pi^{-1}\{\rho\} \cap V_\comp^{\geq2}(T_b)$ with respect to inorder traversal, then $\bw$ is already a canonical generator.
Next, suppose that it is not, and set $\alpha'$ to be the last element of $\pi^{-1}\{\rho\} \cap V_\comp^{\geq2}(T_b)$.
Then we can decompose $\bv_\alpha$ like so:
\begin{align}
\bv_\alpha
&=
b_\rho - \sum_{\gamma \in [\alpha,\beta_\alpha)} a_\gamma
=
\biggl(b_\rho - \sum_{\gamma \in [\alpha',\beta_{\alpha'})} a_\gamma\biggr)
-
\biggl(
\sum_{\gamma \in [\alpha,\beta_\alpha)} a_\gamma
-
\sum_{\gamma \in [\alpha',\beta_{\alpha'})} a_\gamma
\biggr).
\end{align}
\end{proof}

The last piece of preparation we need for the proof of Lemma~\ref{lem:2X_2T_reduced_normal} is the following elementary discrete-geometric result.

\begin{lemma}
\label{lem:ineqs_real_to_int}
Fix $n \in \bZ_{\geq1}$ and $(A_{ij})_{1\leq i < j\leq n}, (B_i)_{1\leq i\leq n}, (C_i)_{1\leq i\leq n} \subset \bZ$, and consider the following system in real variables $x_1,\ldots,x_n$:
\begin{align}
\label{eq:ineqs_real_to_int}
x_i - x_j \geq A_{ij} \:\:\forall\:\: 1\leq i<j\leq n,
\qquad
B_i \leq x_i \leq C_i \:\:\forall\:\: 1\leq i\leq n.
\end{align}
If \eqref{eq:ineqs_real_to_int} has a solution in $\bR^n$, then it has a solution in $\bZ^n$.
\end{lemma}

\begin{proof}
We prove the lemma by induction on $n$.
For $n=1$, the lemma says that if, for $B_1,C_1 \in \bZ$, the inequality $B_1 \leq x \leq C_1$ has a real solution, then it has an integer solution.
This is clear.

Next, suppose that we have proven the claim up to, but not including, some $n \geq 2$.
Choose a real solution $\bx^1 \in \bR^n$ of \eqref{eq:ineqs_real_to_int}.
Then $(x_1^1,\ldots,x_{n-1}^1,B_n)$ also satisfies \eqref{eq:ineqs_real_to_int}.
A vector $(x_1,\ldots,x_{n-1},B_n) \in \bR^n$ satisfies \eqref{eq:ineqs_real_to_int} if and only if the variables $x_1,\ldots,x_{n-1}$ satisfies
\begin{align}
\label{eq:ineqs_real_to_int_smaller}
x_i - x_j \geq A_{ij} \:\:\forall\:\: 1 \leq i < j \leq n-1,
\qquad
\max\{B_i,B_n+A_{in}\} \leq x_i \leq C_i \:\:\forall\:\: 1 \leq i \leq n-1,
\end{align}
so $(x_1^1,\ldots,x_{n-1}^1)$ satisfies \eqref{eq:ineqs_real_to_int_smaller}.
It follows from the inductive hypothesis that \eqref{eq:ineqs_real_to_int_smaller} has an integer solution $\bx^2 = (x_1^2,\ldots,x_{n-1}^2)$.
Then $(x_1^2,\ldots,x_{n-1}^2,B_n)$ is an integer solution of \eqref{eq:ineqs_real_to_int}, so we have proven the inductive hypothesis.
\end{proof}

\medskip

\noindent
{\it Proof of Lemma~\ref{lem:2X_2T_reduced_normal}.}
For any monoid $M$, $\Spec M$ is reduced if and only if $M$ is torsion-free.
To show that $\wt{2X}_{2T}$ is reduced, it is therefore enough to show that $L_{2T}$ is saturated in $\bZ^N$.
Similarly, $\Spec M$ is reduced if and only if $M$ is saturated in its groupification, so to prove that $\wt{2X}_{2T}$ is normal, it suffices to show that $\bZ^N_{\geq0} / L_{2T}$ is saturated in $\bZ^N / L_{2T}$.

\medskip

\noindent
{\bf Step 1:}
We show that $L_{2T}$ is saturated in $\bZ^N$.

\medskip

\noindent
Define $\alpha_0$ to be the last element of $\pi^{-1}(\rho_\root) \cap V_\comp^{\geq2}(T_b)$.
For any $\alpha \in V_\comp^{\geq2}(T_b) \setminus \{\alpha_0\}$, we denote by $\bv_\alpha$ the corresponding canonical generator.
To show that $L_{2T}$ is saturated in $\bZ^N$, we must show that if we have $k\bx = \sum_{\alpha \in V_\comp^{\geq2}(T_b) \setminus \{\alpha_0\}} a_\alpha\bv_\alpha$ for some $k \geq 1$, $\bx \in \bZ^N$, and $(a_\alpha) \in \bZ^{\#V_\comp^{\geq2}(T_b) - 1}$, then there exists a sequence $(b_\alpha)$ of integers such that the equality $\bx = \sum_\alpha b_\alpha\bv_\alpha$ holds.
To do this, we will show that each $a_\alpha$ must be divisible by $k$.

Fix $\rho \in V_\inte(T_s) \setminus \{\rho_\root\}$.
We will show inductively that for each $\alpha \in V_\comp^{\geq2}(T_b) \cap \pi^{-1}\{\rho_\root\}$, $a_\alpha$ is divisible by $k$.
Set $\alpha_1$ resp.\ $\alpha_2$ to be the last resp.\ second-to-last elements of $V_\comp^{\geq2}(T_b) \cap \pi^{-1}\{\rho\}$, with respect to inorder traversal.
As the base case, we will show that $a_{\alpha_1}$ is divisible by $k$.
This follows from the fact that $\bv_{\alpha_1}$ is the only $\bv_\alpha$ which has a nonzero entry in the coordinate corresponding to $\rho$, and the fact that every entry of $\sum_\alpha a_\alpha\bv_\alpha = k\bx$ is divisible by $k$.
From this base case, it follows that every entry of $\sum_{\alpha \neq \alpha_1} a_\alpha\bv_\alpha$ is divisible by $k$.
The vertices $\alpha_1$ and $\alpha_2$ are the only elements of $V_\comp^{\geq2}(T_b) \cap \pi^{-1}\{\rho\}$ which are nonzero in the entry corresponding to $\alpha_1$, so it follows from the previous sentence that $a_{\alpha_2}$ is divisible by $k$.
Proceeding inductively in this fashion, we see that for each $\alpha \in V_\comp^{\geq2}(T_b) \cap \pi^{-1}\{\rho\}$, $a_\alpha$ is divisible by $k$.

It remains to show that for each $\alpha \neq \alpha_0$ in $V_\comp^{\geq2}(T_b) \cap \pi^{-1}\{\rho_\root\}$, $a_\alpha$ is divisible by $k$.
This follows from an argument similar to the one we made in the previous paragraph.

\medskip

\noindent
{\bf Step 2:}
We show that $\bZ^N_{\geq0} / L_{2T}$ is saturated in $\bZ^N / L_{2T}$.

\medskip

\noindent
We must show that if $k\bx + \sum_\alpha a_\alpha\bv_\alpha$ lies in $\bZ^N_{\geq0}$ for some $k \geq 1$, $\bx \in \bZ^N$, and $(a_\alpha)$ a sequence of integers, then there exists another integer sequence $(b_\alpha)$ such that $\bx + \sum_\alpha b_\alpha\bv_\alpha$ lies in $\bZ^N_{\geq0}$.
To do so, we will rephrase the condition that $\bx + \sum_\alpha b_\alpha\bv_\alpha \in \bZ^N$ lies in $\bZ^N_{\geq0}$ by analyzing, for any given $\alpha \in V_\comp(T_b) \setminus \{\alpha_\root\}$ or $\rho \in V_\inte(T_s)\setminus\{\rho_\root\}$, the entry of $\bv_{\alpha'}$ corresponding to $\alpha$ or $\rho$.
\begin{itemize}
\item
Fix $\alpha \in V_\comp(T_b) \setminus \{\alpha_\root\}$.
There there are at most two elements $\alpha' \in V_\comp^{\geq2}(T_b) \setminus \{\alpha_0\}$ such that the index-$\alpha$ entry of $\bv_{\alpha'}$ is nonzero, because inorder traversal meets each vertex of $T_b$ twice.
Moreover, suppose that for some distinct $\alpha'$ and $\alpha''$ in $V_\comp^{\geq2}(T_b) \setminus \{\alpha_0\}$, the index-$\alpha$ entries of $\bv_{\alpha'}$ and $\bv_{\alpha''}$ are nonzero; also, suppose $\alpha' < \alpha''$ with respect to inorder traversal.
Then the index-$\alpha$ entries of $\bv_{\alpha'}$ resp.\ $\bv_{\alpha''}$ are $-1$ resp.\ $1$.

\medskip

\item
Fix $\rho \in V_\inte(T_s) \setminus \{\alpha_\root\}$.
If $\pi^{-1}\{\rho\} \cap V_\comp^{\geq2}(T_b)$ is nonempty, then set $\alpha'$ to be the last element of this set with respect to inorder traversal.
Then $\bv_{\alpha'}$ is the only canonical generator whose index-$\alpha$ entry is nonzero.
Alternately, if $\pi^{-1}\{\rho\} \cap V_\comp^{\geq2}(T_b)$ is empty, there is no canonical generator whose index-$\alpha$ entry is nonzero.
\end{itemize}
It follows from this analysis that we can rephrase the condition that $\bx + \sum_\alpha b_\alpha\bv_\alpha \in \bZ^N$ lies in $\bZ^N_{\geq0}$ as a system of inequalities of the form considered in Lemma~\ref{lem:ineqs_real_to_int}.
The hypothesis that $k\bx + \sum_\alpha b_\alpha\bv_\alpha$ lies in $\bZ^N_{\geq0}$ implies that our system has a rational solution in the $b_\alpha$'s, so by Lemma~\ref{lem:ineqs_real_to_int}, it has an integer solution.
\null\hfill$\triangle$

\medskip

\begin{example}
\label{ex:reduced_and_normal}
We will now demonstrate our proof that the local models are reduced and normal, in the case of a particular tree-pair.
The following figure depicts a tree-pair $2T = T_b \to T_s$, with its gluing parameters labeling the corresponding edges and with colored paths corresponding to the coherence conditions on the gluing parameters.
\begin{figure}[H]
\centering
\def\svgwidth{0.4\columnwidth}
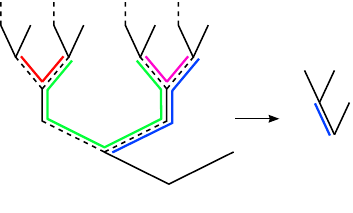
\end{figure}
\noindent
That is, the local model $2X_{2T}$ is the following closed subset of $\bC^7$:
\begin{align}
2X_{2T}
=
\Bigl\{
(a,b,c,d,e,f,A) \in \bC^7
\:\left|\:
{{c=d, \: e=f,}
\atop
{ad = be, \: bf = A}}
\right.\Bigr\}.
\end{align}
As $\Spec$ of a monoid algebra, we can write $2X_{2T}$ like so:
\begin{gather}
2X_{2T}
=
\Spec \bigl(\bZ_{\geq0}^7 / L\bigr),
\quad
L
\coloneqq
\langle \bv_1,\bv_2,\bv_3,\bv_4\rangle,
\quad
\left(\begin{array}{c}
\bv_1
\\
\bv_2
\\
\bv_3
\\
\bv_4
\end{array}\right)
\coloneqq
\left(\begin{array}{ccccccc}
0 & 0 & -1 & 1 & 0 & 0 & 0
\\
-1 & 1 & 0 & -1 & 1 & 0 & 0
\\
0 & 0 & 0 & 0 & -1 & 1 & 0
\\
0 & -1 & 0 & 0 & 0 & -1 & 1
\end{array}\right).
\end{gather}

First, we show that $2X_{2T}$ is reduced, which is equivalent to $L$ being saturated in $\bZ^7$.
This is equivalent to the statement that if we have $k\bx = \sum_{i=1}^4 a_i\bv_i$ for some $k \geq 1$, $\bx \in \bZ^7$, and $a_1,a_2,a_3,a_4 \in \bZ$, then we have $\bx = \sum_{i=1}^4 b_i\bv_i$ for some $b_1,b_2,b_3,b_4 \in \bZ$.
Since $L$ is freely generated by $(\bv_1,\bv_2,\bv_3,\bv_4)$, this is equivalent to the statement that if $\sum_{i=1}^4 a_i\bv_i$ is divisible by $k \geq 1$, then $k$ divides each $a_i$.
The \nth{7} entry of $\sum_{i=1}^4 a_i\bv_i$ is $a_4$, hence $k$ divides $a_4$, and in addition we see that $k$ divides $\sum_{i=1}^3 a_i\bv_i$.
The \nth{6} entry of $\sum_{i=1}^3 a_i\bv_i$ is $a_3$, hence $k$ divides $a_3$.
Proceeding in this fashion, we see that $k$ indeed divides each $a_i$, so $L$ is saturated in $\bZ^7$.

Next, we show that $2X_{2T}$ is normal.
This is equivalent to $\bZ_{\geq0}^7 / L$ being saturated in its groupification $\bZ^7 / L$,
which is in turn equivalent to the condition that if, for some $k \geq 1$, $\bx \in \bZ^7$, and $a_1,a_2,a_3,a_4 \in \bZ$, the vector $k\bx + \sum_{i=1}^4 a_i\bv_i$ lies in $\bZ_{\geq0}^7$, then there exist $b_1,b_2,b_3,b_4 \in \bZ$ such that $\bx + \sum_{i=1}^4 b_i\bv_i$ lies in $\bZ_{\geq0}^7$.
This is equivalent to the statement that if the following system  in $b_1,b_2,b_3,b_4$ has a rational solution, then it has an integer solution:
\begin{gather}
b_1 \leq x_3,
\qquad
b_2 \leq x_1,
\qquad
b_4 \geq -x_7,
\\
b_1-b_2 \geq -x_4,
\qquad
b_2-b_3 \geq -x_5,
\qquad
b_2-b_4 \geq -x_2,
\qquad
b_3-b_4 \geq -x_6.
\nonumber
\end{gather}
This implication follows from Lemma~\ref{lem:ineqs_real_to_int}.
\null\hfill$\triangle$
\end{example}

\begin{proof}
We prove the lemma by induction on $n$.
For $n=1$, the lemma says that if, for $B_1,C_1 \in \bZ$, the inequality $B_1 \leq x \leq C_1$ has a real solution, then it has an integer solution.
This is clear.

Next, suppose that we have proven the claim up to, but not including, some $n \geq 2$.
Choose a real solution $\bx^1 \in \bR^n$ of \eqref{eq:ineqs_real_to_int}.
Then $(x_1^1,\ldots,x_{n-1}^1,B_n)$ also satisfies \eqref{eq:ineqs_real_to_int}.
A vector $(x_1,\ldots,x_{n-1},B_n) \in \bR^n$ satisfies \eqref{eq:ineqs_real_to_int} if and only if the variables $x_1,\ldots,x_{n-1}$ satisfies
\begin{align}
\label{eq:ineqs_real_to_int_smaller}
x_i - x_j \geq A_{ij} \:\:\forall\:\: 1 \leq i < j \leq n-1,
\qquad
\max\{B_i,B_n+A_{in}\} \leq x_i \leq C_i \:\:\forall\:\: 1 \leq i \leq n-1,
\end{align}
so $(x_1^1,\ldots,x_{n-1}^1)$ satisfies \eqref{eq:ineqs_real_to_int_smaller}.
It follows from the inductive hypothesis that \eqref{eq:ineqs_real_to_int_smaller} has an integer solution $\bx^2 = (x_1^2,\ldots,x_{n-1}^2)$.
Then $(x_1^2,\ldots,x_{n-1}^2,B_n)$ is an integer solution of \eqref{eq:ineqs_real_to_int}, so we have proven the inductive hypothesis.
\end{proof}

\section{Further directions: wonderful compactifations }
\label{sec:wond-comp-}

There  are several interpretations of the moduli space \(\ol M_r\) as an iterated blow up \cite{keel,fm}.
In the work \cite{cgk} an iterated blow up construction is provided for the moduli space of the stable pointed projective spaces.
The work of Li Li \cite{Li09} explains how all of these constructions are particular cases of wonderful compactification spaces.

In this section we outline a construction for the space \(\ol{2M}_\bn^\bC\) that is analogous to the construction of \cite{fm}.
The main goal for us is to construct
the relative version of the Fulton-MacPherson spaces for smooth maps \(g:X\to Y\) of the smooth varieties. This section does not contains any proofs since we postponing
the details for the forth-coming publication.

The relative versions of the Fulton--MacPherson spaces are not smooth in general, but their singularities are at worst toric.
Thus if want to have an iterative blow up construction of the relative \fm spaces we need to generalize the construction from \cite{fm} to the setting where we allow toric singularities of the spaces.
We rely on the wonderful compactification approach to the \fm spaces \cite{Li09} in our generalization.
We also construct the relative version of Ulyanov spaces \cite{ul} and we conjecture that the Ulyanov spaces provide a resolution for the singularities of the relative \fm spaces.

We start with the conventions and notations for the diagonals and the multi-diagonals in the relative setting.
In the next subsection we formulate a weak version of the transversality condition from the wondreful blow-up construction of \cite{Li09}.

\subsection{Relative diagonals}

Let us assume that the labels \(\mathbf{n}\) are ordered, \(n_1\ge n_2\ge \dots\ge n_\ell\) and \(n_i>0\) for \(i\le \ell^+(\mathbf{n})\), \(n_i=0\) for \(i>\ell^+(\mathbf{n})\).
In the previous formulas we used \(|\mathbf{n}|=\sum_i n_i\) and later we also use \(\|\mathbf{n}\|=|\mathbf{n}|+\ell(\mathbf{n})-\ell^+(\mathbf{n}).\)
We use the following notation for finite sets: \([k]\) stands for the set \(\{1,\dots,k\}\), and \([k,l]\) stands for \(\{k,\dots,l\}\).
For a set \(X\) and \(I\subset [n]\) we set \(\Delta_X(I)\subset X^n\) to be the corresponding diagonal.

If \(I\subset [\|\mathbf{n}\|]\) then \(I^{\mathbf{n}}_Y\subset [\ell(\mathbf{n})]\) is the subset defined by
\begin{align}
I^{\mathbf{n}}_Y
=
\bigl\{i
\:\big|\:
i\in [\ell^+(\mathbf{n})], I\cap [n_i,n_{i+1}-1]\ne \emptyset\bigr\}
\cup
\bigl\{i
\:\big|\:
i\in [\ell^+(\mathbf{n})+1,\ell(\mathbf{n})],
  i-\ell^+(\mathbf{n})\in I-|\mathbf{n}|\bigr\}.
  \end{align}
Similarly, we define \(I^{\mathbf{n}}_X=I\cap [|\mathbf{n}|]\). 

In this subsection we present a blow-up construction of the spaces from the previous sections.
To be precise, we work with a smooth map \(g:X\to Y\) and we use the notation \((X/Y)^{\mathbf{n}}\) for the push-out:
\[\begin{tikzcd}
    X^{|\mathbf{n}|}\arrow[r,"g"]&Y^{|\mathbf{n}|}\\ (X/Y)^{\mathbf{n}}\arrow[r,"\pi_Y"]\arrow[u,"\pi_X"] & Y^{\ell(\mathbf{n})}\arrow[u,"\delta(\mathbf{n})"].
  \end{tikzcd}
\]
Here we used the notation \(\delta(\bn)\) for the map of the multi-diagonals corresponding to \(\bn\):
the image of map \(\delta(\mathbf{n})\) is the multidiagonal
\begin{align}
\Delta_Y(1,\dots,n_1)\cap\dots\cap\Delta_Y(n_1+n_2+\dots+n_{\ell^+-1}+1,\dots,n_1+\dots +n_{\ell^+}),\end{align}
\(\ell^+=\ell^+(\mathbf{n})\), and the map \(\delta(\mathbf{n})\) projects along the \(Y\)-factor that is labeled by \(0\).

The relative diagonals \(\Delta_{X/Y}(I)\subset (X/Y)^{\mathbf{n}}\) are naturally labeled by the subset \(I\subset [\|\mathbf{n}\|]\):
\begin{align}
\Delta_{X/Y}(I)=\pi_Y^{-1}(\Delta_Y(I^{\mathbf{n}}_Y))\cap \pi_X^{-1}(\Delta_X(I^{\mathbf{n}}_X)).
\end{align}
The relative multi-diagonals are intersections of the diagonals we described above.
  


\subsection{Toroidal wonderful compactifications}
\label{sec:toro-wond-comp}

Let us recall and slightly generalize some terminology from the paper \cite{Li09}. 
We also introduce a mild extension of the main result of the mentioned paper.

Fix a toroidal variety \(Z\), i.e.\ a variety with toric singularities.
An \emph{arrangement} of subvarieties \(\cS\) is a finite collection of nonsingular subvariaties such that all nonempty scheme-theoretic intersections of subvarieties are again in \(\cS\).
Equivalently, we can require any pairwise intersection to be clean.

Let \(Z\) be smooth.
A subset \(\mathcal{G}\subset \mathcal{S}\) is then called a \emph{building set} if, for all \(S\in\cS\setminus\cG\), the set of minimal elements in \(\underline{S}=\{G\in \cG \:|\: G\supset S\}\) intersect transversely and the intersection is \(S\).
Up to this point we were repeating the definitions of \cite{Li09} and now we introduce a new object.

Now let us allow \(Z\) to have toric singularities.
We can no longer talk about transverse subvarieties, and we need to modify the previous definition.
A subset \(\cG \subset \cS\) is called a \emph{toroidal building set} if, for all \(S\in\cS\setminus\cG\), the set of minimal elements in \(\underline{S}\) intersect scheme-theoretically at \(S\) and any point \(s\in S\) has a neighborhood where \(\cup_{G\in \underline{S}} G\) is defined by a monomial ideal\footnote{Since locally \(Z\) is modeled on a possibly-singular toric variety, the notion of a monomial ideal is well-defined.}.
More generally, we call the collection of subsets \(\cG\) a building set if the intersection of the subsets of the collection form an arrangement and \(\cG\) is a building set for this arrangement.

Let us denote by \(\Delta_{X/Y}\) the collection of all relative diagonals of \((X/Y)^{\mathbf{n}}\) and by \(\Delta_{X/Y}^\bullet\) the collection of all relative multidiagonals.

\begin{lemma}
For any \(\mathbf{n}\), the collection \(\Delta_{X/Y}^\bullet\) is an arrangement and \(\Delta_{X/Y}\) is a toroidal building set for it.
\end{lemma}

\begin{definition}
Let \(\cG\) be a nonempty toroidal building set for \(\cS\) and \(Z^\circ=Z\setminus \bigcup_{G\in \cG} G\).
The closure \(Z_\cG\) of the image of the natural locally-closed embedding
\begin{align}
Z^\circ\hra \Pi_{G\in \mathcal{G}}\Bl_GZ
\end{align}
is a \emph{toroidal wonderful compactification} of the arrangement \(\cS\).
\end{definition}

If we replace the toroidal building set in the definition above by the building set, we get a definition of the wonderful compactification from \cite{Li09}.
It is shown in \cite{Li09} that wonderful compactifications can be constructed as iterated blowups.
There is a similar construction for wonderful toroidal compactifications, and we explain it below.

Let \(\pi\colon \Bl_YZ\to Z\) be a blowup of an irreducible toroidal variety \(Z\) along \(Y\subset Z\), and let \(V\subset Z\) be a subvariety.
We define the \emph{dominant transform} \(\wt V\subset \Bl_YZ\) to be a strict transform of \(V\) if \(V\not\subset Y\) and \(\pi^{-1}(V)\) otherwise.

\begin{theorem}
Let \(Z\) be a toroidal variety and let \(\cG = \{G_1,\dots,G_N\}\) be a toroidal building set of subvarieties of \(Z\).
Let \(\mathcal{I}_i\) be the ideal sheaf of \(G_i\).
\begin{enumerate}
\item
The wonderful toroidal compactification \(Z_{\mathcal{G}}\) is isomorphic to the blow-up of \(Z\) along the ideal sheaf \(\mathcal{I}_1\mathcal{I}_2\dots
    \mathcal{I}_N\).

\item
The wonderful toroidal compactification \(Z_{\mathcal{G}}\) is a toroidal variety.

\item
If we arrange \(\{G_1,\dots,G_N\}\) in such an order that the first \(i\) terms \(G_1,\dots, G_i\) form a toroidal building set for any \(1\le i\le N\), then
\begin{align}
Z_{\mathcal{G}}=Bl_{\widetilde{G}_N}\dots Bl_{\widetilde{G}_2}Bl_{\widetilde{G}_1}Z.
\end{align}
\end{enumerate}
\end{theorem}

\subsection{Examples}

The main examples of the  wonderful toroidal compactifications come from the arrangements discussed above.

\begin{enumerate}
\item
Suppose \(X\), \(Y\), and the map \(g\colon X\to Y\) are smooth.
Then \((X/Y)[\bn]\coloneqq(X/Y)^\bn_{\Delta_{X/Y}}\) is a toroidal variety.
If \(g\) is an identity map and \(\bn=\{1^n\}\) then \((X/Y)[\bn]\) is the Fulton--MacPherson space \(X[n]\).

\item
Under the same assumption as before, the compactification \((X/Y)\langle \mathbf{n}\rangle:=(X/Y)_{\Delta_{X/Y}^\bullet}\) is actually smooth and provides a resolution of the variety \((X/Y)[\bn]\).
If \(g\) is the identity map, then the space \((X/Y)\langle \bn\rangle\) is Ulyanov's polydiagonal compactification of \(X^n\).
\end{enumerate}

To  connect our constructions with the objects discussed in the rest of this paper, we observe that the map \(\pi_Y\) induces a map of varieties \(\pi_Y\colon (X/Y)[\mathbf{n}]\to Y^{\ell(\mathbf{n})}\).
Similarly, we define \(\pi_X\colon (X/Y)[\bn]\to X^{|\\bn|}\).
For a point \(y\in Y\) we denote by
\(\Delta_Y(y)\in Y^{\ell(\bn)}\) the corresponding point on the small diagonal.
Similarly, we define \(\Delta_X(x)\in X^{|\bn|}\) for \(x\in X\).
We denote by \((X/Y)[\bn]_x\) the preimage \(\pi_X^{-1}(\Delta_X(x))\cap \pi_Y^{-1}(\Delta_Y(g(x)))\).

\begin{theorem}
Let \(X=\bC^2\) and \(g\colon \bC^2\to \bC\) be the linear projection.
Then
\begin{align}
(\bC^2/\bC)[\bn]_o=\ol{2M}_\bn^\bC.
\end{align}
Here \(o=(0,0)\).
\end{theorem}

Similarly, we define the punctual version of the relative \fm spaces \((X/Y)[\bn]_o\), \(o\in X\) for a smooth map \(g\colon X\to Y\) as well as the punctual version of Ulyanov space \((X/Y)\langle \bn\rangle_o\).
In our forthcoming paper \cite{bo:fm} we prove the following theorem.
\begin{theorem}
The space \((X/Y)[\bn]_o\) is a normal lci variety with at most toric singularities, and the natural blowdown map \((X/Y)\langle \mathbf{n}\rangle_o\to (X/Y)[\mathbf{n}]_o\) is a resolution of singularities.
\end{theorem}

\subsection{A singular example}
\label{sec:singular-example}

In this subsection we analyze the smallest singular example of the variety \((\bC^2/\bC)[\bn]\).
To be more precise, the argument of this subsection could be used to show that the space \((\bC^2/\bC)[4,0]_o\) has only three singular points and that these points are simple quadratic singularities.

The  relative \fm space in question is a blowup along the ideal of the union of the union of the relative diagonals:
\begin{align}
(\bC^2/\bC)[4,0]\subset \bC^6\times \prod_{S\subset [5]}\bP_S^{|S|-2}.
\end{align}
We use notation for the projection \(\pi_\cS\colon (\bC^2/\bC)[4,0]\to \bC^6\times \prod_{S\in \cS}\bP_S^{|S|-2}\), where
\(\cS\) is a collection of subsets inside \([5]\). 

To  discuss the scheme structure of this variety we introduce coordinates on the space \(\bC^6\) in the product:
the heights of the four points on the first line are \(y_1,y_2,y_3,y_4\) and the positions of the two lines are \(x_1,x_5\).
In particular, two relative diagonals \(\Delta_{\bC^2/\bC}(1,2,5)\) and \(\Delta_{\bC^2/\bC}(3,4,5)\) are defined by:
\begin{align}
y_{12}=0,x_{15}=0\quad y_{34}=0, x_{15}=0.
\end{align}

These two relative diagonals are of codimension \(2\), and their intersection is of codimension \(3\).
Thus the intersection is not transverse; it also is not a diagonal.
Let us define an open subset \(U\subset \bC^6\) by the conditions \(y_i\ne y_j\), \(i\in \{1,2\}\), \(j\in \{3,4\}\).
Respectively, the \(\wt U\subset (\bC^2/\bC)[4,0]\) is the preimage of \(U\) under the projection map \(\pi_\emptyset\).

The projection \(\pi_{\{1,2,5\},\{3,4,5\}}\) is an isomorphism when restricted to the open subset \(\wt U\).
The image of the projection is cut out by the following equations:
\begin{align}
\xi_0x_{15}=\xi_1y_{12},\quad \xi'_0x_{15}=\xi'_1y_{34}.
\end{align}

Hence the affine chart \(\xi_0\ne 0\), \(\xi'_0\ne 0\) of the image is isomorphic to the hypersurface inside \(\bC^4\times \bC^3=\bC^7\) defined by the equation \(zy_{12}=wy_{34}\) where \(z=\xi_1/\xi_0\), \(w=\xi'_1/\xi'_0\), \(y_{12},y_{34}\) are the coordinates along the first factor of \(\bC^7\) and \(y_1,x_1\) and \(y_{23}\) are the coordinates along the last factor.

The previous discussion implies that that the subspace \((\bC^2/\bC)[4,0]_o\) could not be smooth.
To describe the singularity in  more details, we recall that \((\bC^2/\bC)[4,0]\) is defined in \(\bC^6\times \prod_{S\subset [5]}\mathbb{P}_S^{|S|-2}\) by the following equation:
\begin{align}
\label{eq:gl-eq}
\xi^{S}_{ij}w_{kl}=\xi^S_{kl}w_{ij},
\end{align}
where \(\xi_{ij}^S\) are homogeneous coordinates on \(\bP^{|S|-2}_S\) that satisfy the relations \(\xi_{ij}^S+\xi_{jk}^S=\xi_{ik}^S\), \(\xi_{ii}^S=0\), and \(w_{ij}=y_i-y_j\), \(w_{i5}=x_1-x_5\), \(i,j\in[4]\) satisfy similar relations.

In these coordinates we define an open subset \(V\subset (\bC^2/\bC)[4,0]\) by the inequalities
\begin{align}
\label{eq:open}
\xi^S_{ij}\ne 0, \quad i\in \{1,2\}, j\in\{3,4\}, \quad \xi_{12}^{125}\ne 0, \xi_{34}^{345}\ne 0.
\end{align}
Using \eqref{eq:gl-eq} we express \(\xi^S_{ij}\), \(x_i,y_i\) in terms of \(x_{12},y_{23},x_1,y_1\) and \(\xi^{[4]}_{ij}/\xi^{[4]}_{23}\), \(\xi^{1,2,5}_{15}/\xi^{1,2,5}_{12}\), \(\xi^{3,4,5}_{35}/\xi^{3,4,5}_{34}\) and the expressions are regular on the open subset defined by the inequalities (\ref{eq:open}).
Thus the projection of \(V\subset (\bC^2/\bC)[4,0]\) to \(\mathbb{P}_{125}^1\times\mathbb{P}_{345}^1\times\mathbb{P}^2_{1234}\times \bC^4\), where the coordinates along the last factor are \(x_1,y_1,x_{12},y_{23}\), is an isomorphism.

Finally, let us observe that the above projection sends the subvariety \((\bC^2/\bC)[4,0]_o\cap V\) to the locus defined by
\begin{align}
x_{12}=y_{23}=x_1=y_1=0, \quad \xi^{125}_{15}\xi^{[4]}_{12}=\xi^{345}_{35}\xi^{[4]}_{34}.
\end{align}
Thus the point of the locus with \(\xi_{15}^{125}=\xi_{12}^{[4]}=\xi_{35}^{345}=\xi_{34}^{[4]}=0\) is an isolated quadratic singularity.

This singularity matches with the singularity of the moduli space \(\ol{2M}_{4,0}^{\bC}\) that was studied in Ex.~\ref{ex:reduced_and_normal}.

\subsection{Strata structure} 
\label{sec:warm-up}

In this subsection we match the natural stratification of the spacre \(\ol{M}_r^\bC\) by the open strata \(Z_T\) with a natural stratification of the blow-up descrioptiuon of this moduli space.
The analogue if this match for \(\ol{2M}_n\) as well as discussion of the charts around these strata is discussed in our next publication \cite{bo:fm}.


The wonderful compactification for the space  \(\ol{M}_r^\bC\) relies on the properties of the space \(\bC^1[r]\), and we provide an explicit construction for the space below.
The ambient space for it is the product \(X_\amb=\bC^{r}\times \prod_{S\subset [r]}\bP^{|S|-2}_S\).
The coordinates on the first factor we set to be \(z_i\), \(i\in [r]\), and the homogeneous coordinates on \(\bP^{|S|-2}_S\) are \(\xi_{ij}^S\), \(i,j\in S\) modulo the natural linear relations \(\xi^S_{ij}=-\xi_{ji}^S\), \(\xi^S_{ij}+\xi^S_{jl}=\xi^S_{il}\).
We also introduce coordinates \(\xi^S_i\), \(i\in S\), which are defined up to an affine transformation.
Similarly, we use the notation \(z_{ij}=z_i-z_j\).

The defining ideal of the space is constructed in two steps.
We define the ideal \(I'\subset R=\bC[x,\xi]\) as an ideal generated by the following elements:
\begin{align}
z_{ij}\xi^{S}_{kl}-z_{kl}\xi_{ij}^{S}.
\end{align}
The localized ideal \(I_{loc}\) is the ideal inside \(\bC(z)\otimes \bC[\xi]\) and is defined as \(I'\otimes_R\bC(z)\otimes \bC[\xi]\).
In the second step we construct the  saturated ideal \(I_\sat=I'\cap \bC[z,\xi]\) and the defining ideal of our space is defined as the restriction \(I=I_\sat|_{z_1=\dots=z_r=0}\).
That is,
\begin{align}
\bC^1[r]=\mathrm{Spec}(\bC[z,\xi]/I_{sat}),\quad \ol{M}_r^{\bC}=\mathrm{Spec}(\bC[\xi]/I).
\end{align}

\begin{example}
If \(r=4\) then \(X_\amb=\bC^4\times\bP^2\times \bP^1\times \bP^1\times \bP^1\times \bP^1\) and as a result of the above construction one obtains a description of \(\ol M_4^\bC\subset \bP^2\times\bP^1\times\bP^1\times\bP^1\times\bP^1\) as a blowup of \(\bP^2=\bP_{S}^2\), \(S=\{1,2,3,4\}\) at the four points \(p_{S',S''}=\{\xi^S_{S'}=\xi^S_{S''}=0\}\), \(S'\cup S''=\{1,2,3,4\}\), \(|S'|=|S''|=2\).
\end{example}

Now let us compare the above description with the  model from the previous sections.
Let us recall the  sets \(X_T\) in the definition \ref{def:Mr_local_model} stratify the space \(\ol M_r^\bC\).
Let \(\sB\) the 1-bracket corresponding to \(T\).
Then the closure of the set \(\ol{X_T}\) in the wonderful compactification model is defined by the following collection of equations.

There is a natural correspondence between the vertices of the tree \(T\) and the sets in \(\sB\): \(v\mapsto S_v\in \sB\).
Since the leaves of the tree \(T\) are labeled with integers from \(1\) to \(r\), descend along the edges of \(T\) provides us with the map \(\phi^T_v: S_v\to \incom(v)\).
Thus the defining ideal \(I_{Z_T}\) of \(\ol{X_T}\) is generated by:
\begin{align}
\xi^S_i-\xi^S_j, \quad S=S_v\in \sB, \quad \phi^T_v(i)=\phi^T_v(j).
\end{align}

The analogue of the strata correspondence together with the description of the analogue of the atlas from the section \ref{ss:2Mn-bar_charts} for the relative \fm spaces will appear in the sequel to this paper.


\appendix

\section{Examples}
\label{s:examples}

In this appendix, we exhibit all 1- resp.\ 2-dimensional instances of $\ol{2M}_\bn^\bC$ as $\bP^1$ resp.\ blowups of $\bP^1 \times \bP^1$.

\begin{figure}[H]
\centering
\def\svgwidth{0.9\columnwidth}
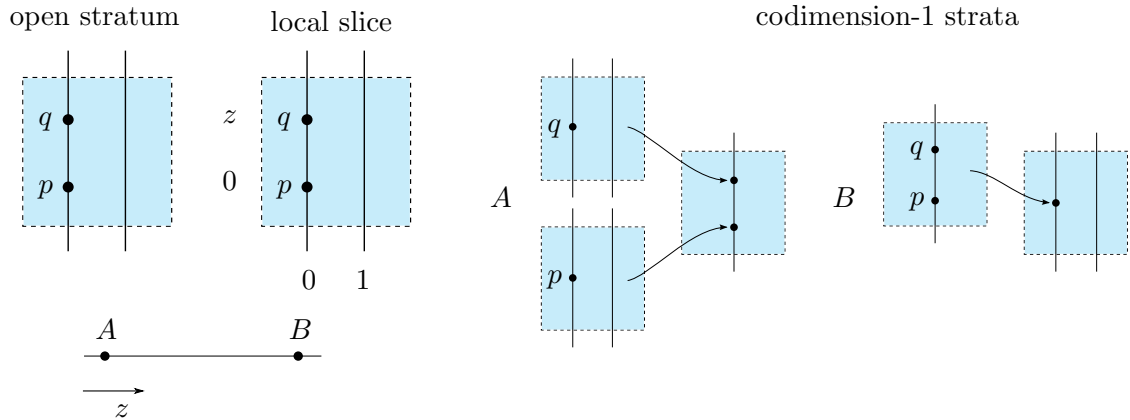
\caption{
In this figure, we demonstrate why $W_{20}^\bC$ is isomorphic to $\bP^1$.
First, we consider a particular local slice for the open stratum, which identifies this locus with $\bC \setminus \{0\} \ni z$.
Next, we record the two codimension-1 strata.
We see that these strata correspond to the missing points $z=0$, $z=\infty$.
}
\end{figure}

\begin{figure}[H]
\centering
\def\svgwidth{0.7\columnwidth}
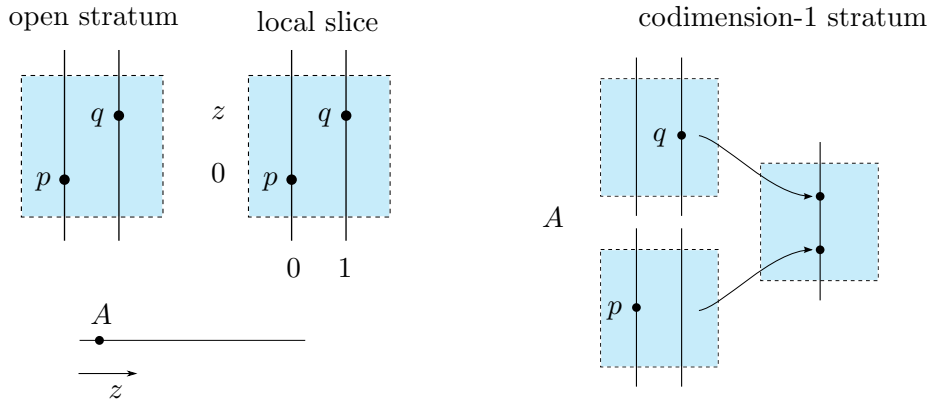
\caption{
In this figure, we demonstrate why $W_{11}^\bC$ is isomorphic to $\bP^1$.
First, we consider a particular local slice for the open stratum, which identifies this locus with $\bC \ni z$.
Next, we record the one codimension-1 stratum.
We see that this stratum corresponds to the missing point $z=\infty$.
}
\end{figure}

\newpage

\begin{figure}[H]
\centering
\captionsetup{width=\columnwidth}
\def\svgwidth{1.0\columnwidth}
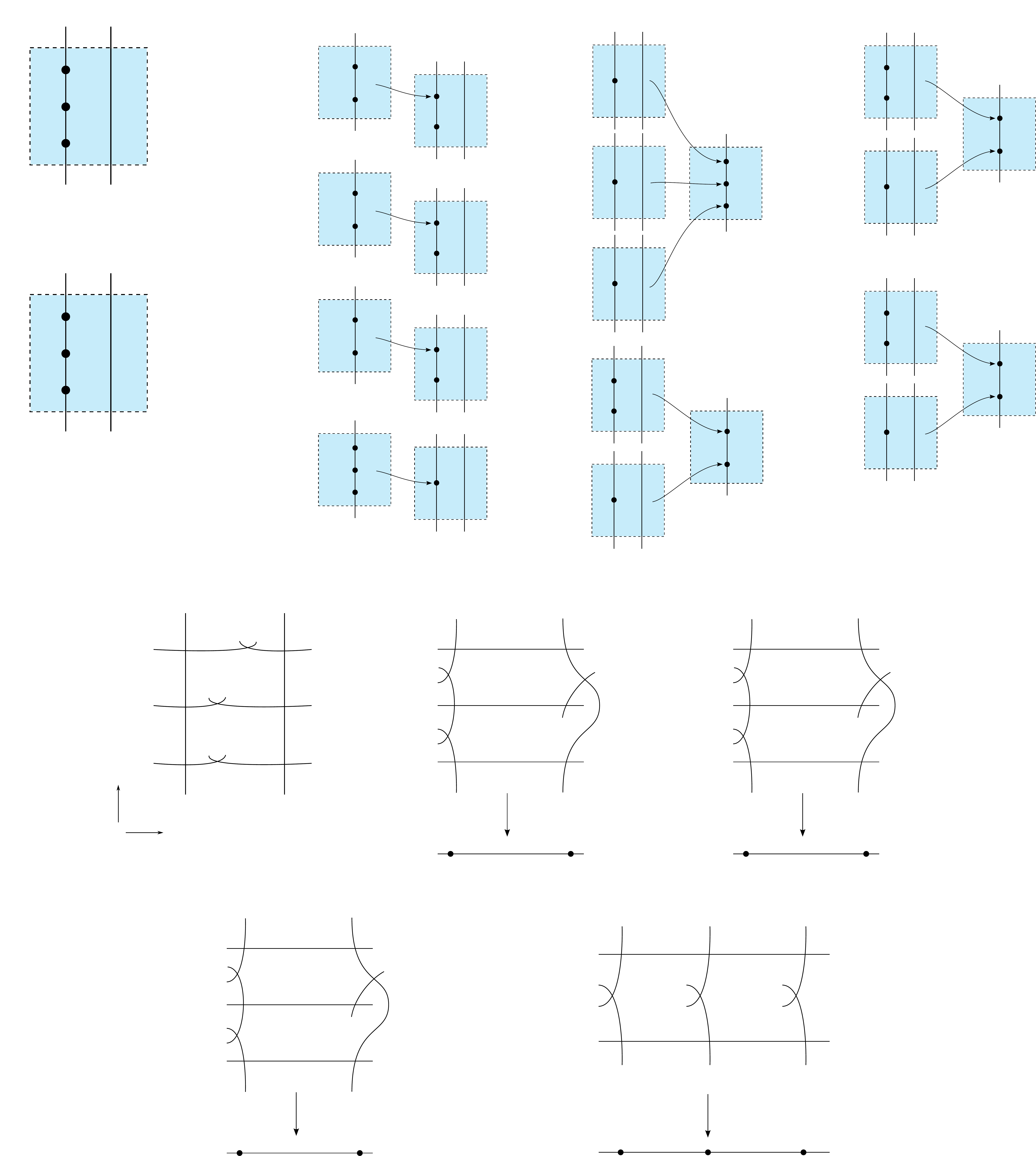
\caption{
In this figure, we demonstrate why $W_{30}^\bC$ is isomorphic to $\Bl_{(0,0),(0,1),(\infty,\infty)} (\bP^1 \times \bP^1)$.
First, we consider a particular local slice for the open stratum, which identifies this locus with $(\bC\setminus\{0\})\times(\bC\setminus\{0,1\}) \ni (z,w)$.
Next, we record the eight codimension-1 strata.
We see that five of these strata ($E$, $D$, $B$, $A$, $G$) correspond to $z=0$, $z=\infty$, $w = 0$, $w=1$, $w=\infty$, and three ($H$, $F$, $C$) correspond to the exceptional divisors arising from blowup at the points $(0,0)$, $(0,1)$, $(\infty,\infty)$.
Finally, we illustrate the forgetful maps to $\bCP^1$ that result from forgetting either $p$, $q$, $r$, or the line carrying no marked points.
}
\end{figure}

\newpage

\begin{figure}[H]
\centering
\def\svgwidth{1.0\columnwidth}
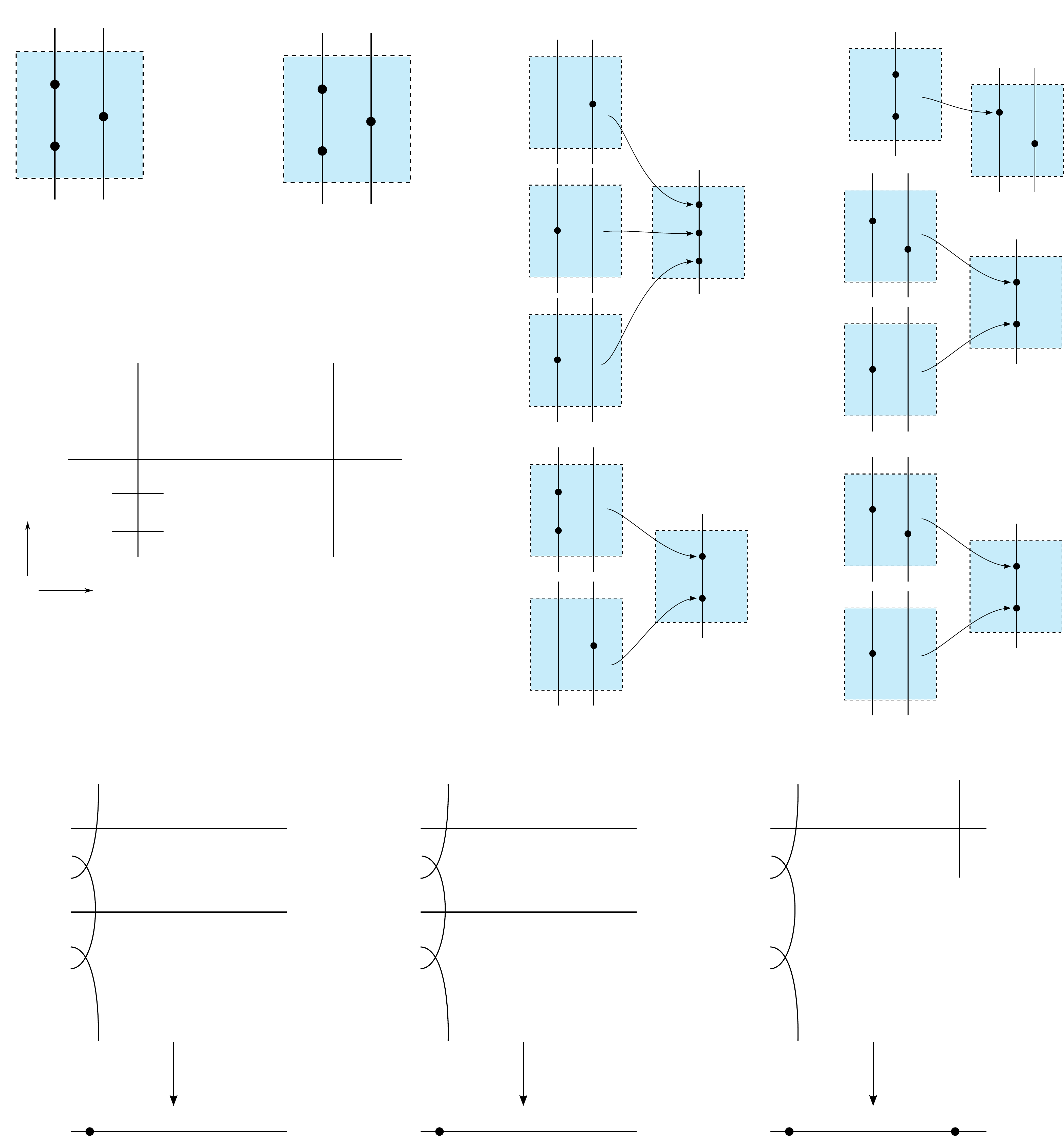
\caption{
In this figure, we demonstrate why $W_{21}^\bC$ is isomorphic to $\Bl_{(0,0),(0,1)} (\bP^1 \times \bP^1)$.
First, we consider a particular local slice for the open stratum, which identifies this locus with $(\bC\setminus\{0\})\times\bC \ni (z,w)$.
Next, we record the eight codimension-1 strata.
We see that five of these strata ($B$, $A$, $C$) correspond to $z=0$, $z=\infty$, $w=\infty$, and two ($D$, $E$) correspond to the exceptional divisors arising from blowup at the points $(0,0)$, $(0,1)$.
Finally, we illustrate the forgetful maps to $\bCP^1$ that result from forgetting either $p$, $q$, or $r$.
}
\end{figure}

\newpage

\begin{figure}[H]
\centering
\def\svgwidth{0.95\columnwidth}
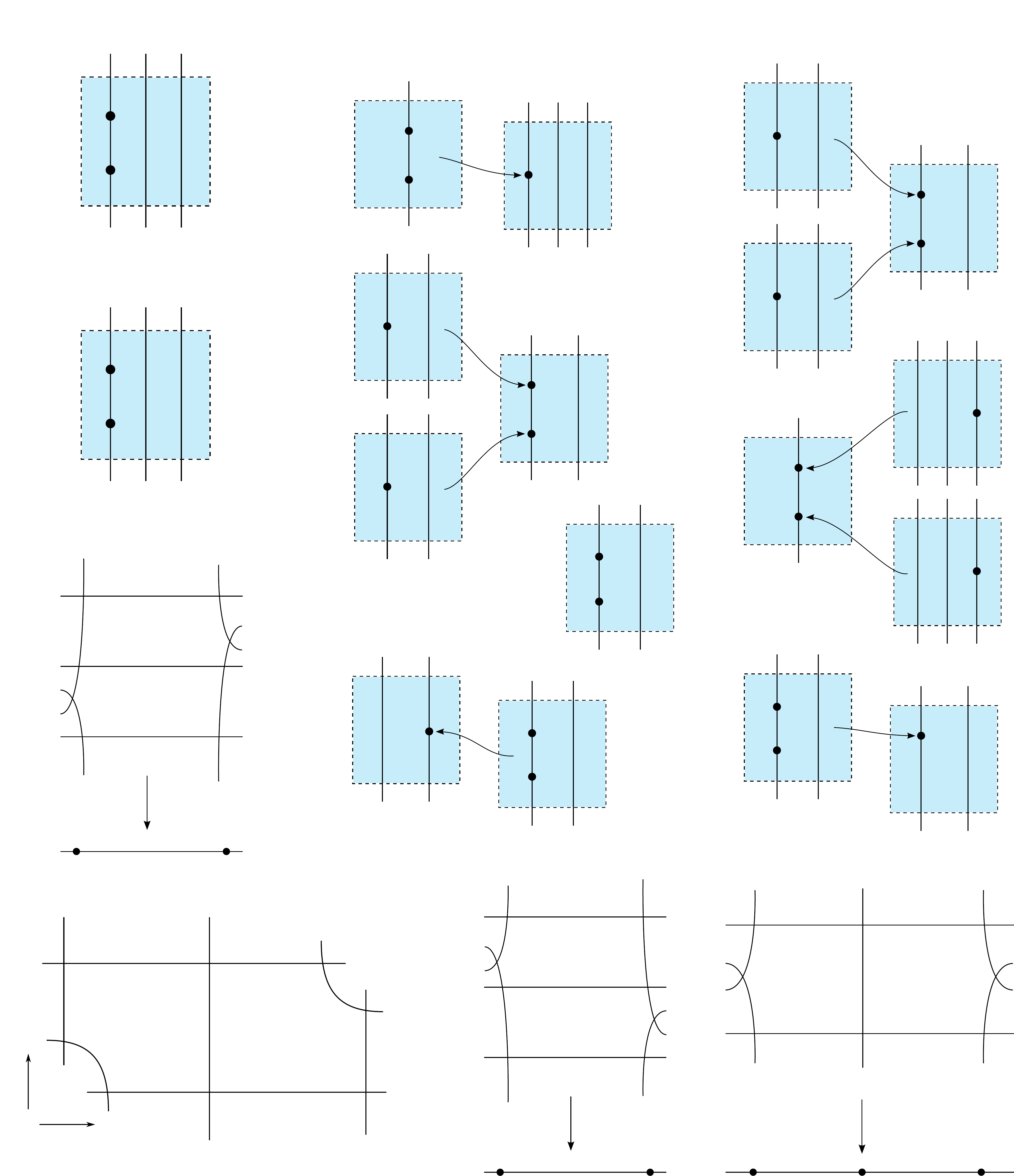
\caption{
In this figure, we demonstrate why $W_{200}^\bC$ is isomorphic to $\Bl_{(0,0),(\infty,\infty)} (\bP^1 \times \bP^1)$.
First, we consider a particular local slice for the open stratum, which identifies this locus with $(\bC\setminus\{0,1\})\times(\bC\setminus\{0\}) \ni (z,w)$.
Next, we record the seven codimension-1 strata.
We see that five of these strata ($D$, $G$, $C$, $A$, $F$) correspond to $z=0$, $z=1$, $z=\infty$, $w=0$, $w=\infty$, and two ($E$, $B$) correspond to the exceptional divisors arising from blowup at the points $(0,0)$, $(\infty,\infty)$.
Finally, we illustrate the forgetful maps to $\bCP^1$ that result from forgetting a point $p$ or $q$, or a line $K$ or $R$.
}
\end{figure}

\newpage

\begin{figure}[H]
\centering
\def\svgwidth{1.0\columnwidth}
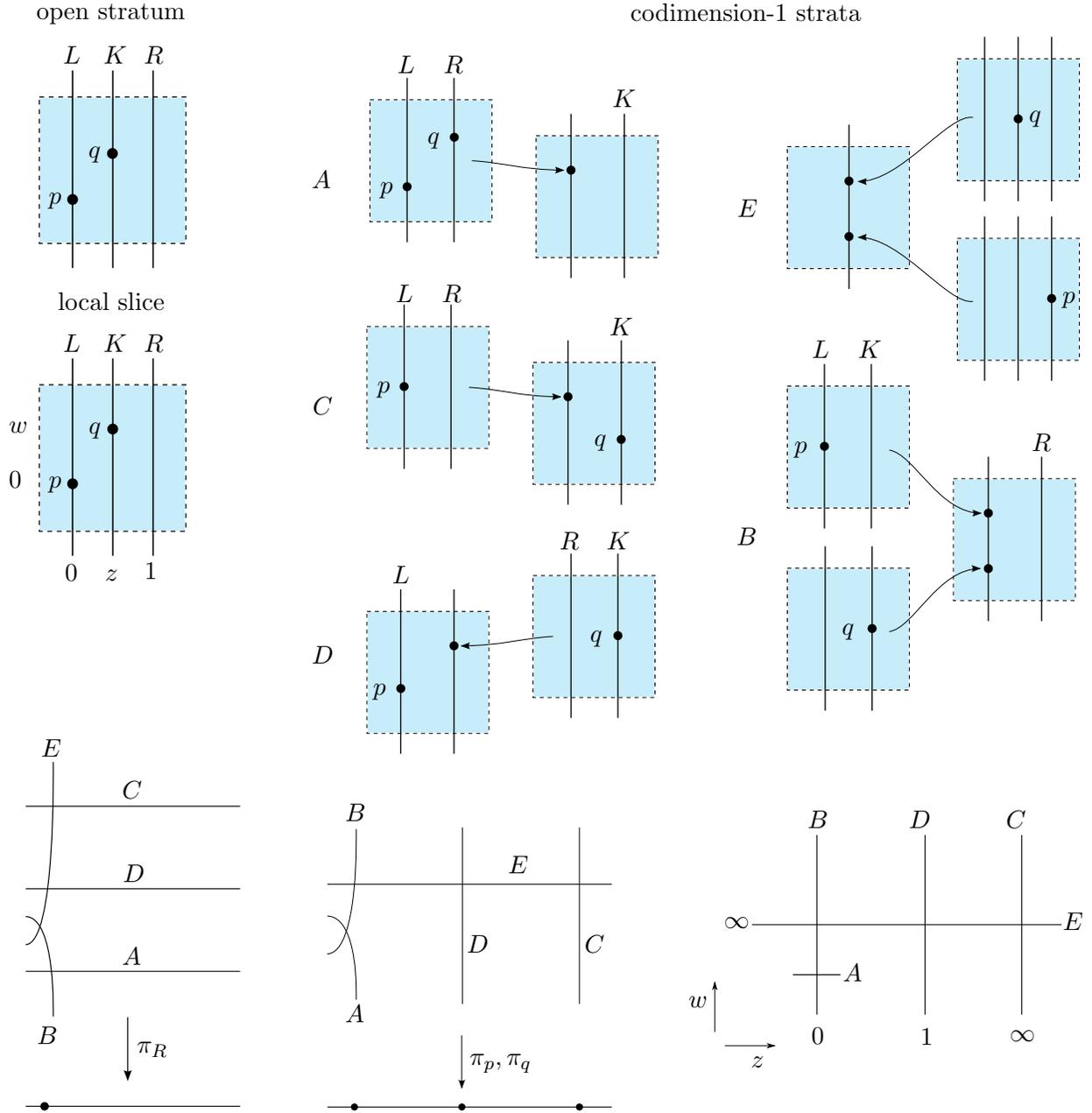
\medskip
\caption{
In this figure, we demonstrate why $W_{110}^\bC$ is isomorphic to $\Bl_{(0,0)} (\bP^1 \times \bP^1)$.
First, we consider a particular local slice for the open stratum, which identifies this locus with $(\bC\setminus\{0,1\})\times\bC \ni (z,w)$.
Next, we record the five codimension-1 strata.
We see that four of these strata ($B$, $D$, $C$, $E$) correspond to $z=0$, $z=1$, $z=\infty$, $w=\infty$, and one ($A$) corresponds to the exceptional divisor arising from blowup at the point $(0,0)$.
Finally, we illustrate the forgetful maps to $\bCP^1$ that result from forgetting a point $p$ or $q$, or the line $R$.
}
\end{figure}

\newpage

\section{An algorithm to compute the virtual Poincar\'{e} polynomial}
\label{s:vpp}

In this appendix, we explain a recursive algorithm for computing the virtual Poincar\'{e} polynomial of $\ol{2M}_\bn^\bC$.
The stratification of $\ol{2M}_\bn^\bC$ will play a crucial role.

Recall that we can associate to any complex variety $X$ its \emph{virtual Poincar\'{e} polynomial} $P[X](t)$, which can be constructed by considering a mixed Hodge structure on the compactly-supported cohomology of $X$ (see \cite{fulton:toric}).
$P[X]$ has the following properties:
\begin{itemize}
\item
If $X$ is smooth and projective, then $P[X]$ is the ordinary Poincar\'{e} polynomial:
\begin{align}
\label{eq:PP_VPP}
P[X](t)
=
\sum_i \rk H^i(X) \cdot t^i.
\end{align}

\item
If $Y \subset X$ is closed and algebraic, then the following identity holds:
\begin{align}
\label{eq:VPP_decomposition}
P[X](t)
=
P[Y](t) + P[X\setminus Y](t).
\end{align}
\end{itemize}
In this subsection, we will explain an algorithm for recursively computing $P[\ol{2M}_\bn^\bC]$, which we have implemented in \textsc{Python} (see \cite{bo:program}).
In Table~\ref{tab:VPP}, we demonstrate this algorithm by listing the virtual Poincar\'{e} polynomials of the 2- and 3-dimensional instances of $\ol{2M}_\bn^\bC$.
On a laptop computer, \cite{bo:program} returns the virtual Poincar\'{e} polynomial of a 6- resp.\ 7-dimensional instance of $\ol{2M}_\bn^\bC$ in $\sim$2 resp.\ $\sim$10 seconds.

We begin by describing the idea of our algorithm.
For any $m$-dimensional complex variety $X$, to compute $P[X]$ it suffices by \eqref{eq:VPP_decomposition} to produce a decomposition $X = \bigcup_{d=0}^m X_d$, where $X_d$ is $d$-dimensional, the inclusions $X_0 \subset \cdots \subset X_m$ hold, $X_{d-1}$ is closed in $X_d$, and we can compute $P[X_0]$ and $P[X_d\setminus X_{d-1}]$ for each $d$.
Indeed, in this situation, we can compute $P[X]$ as
\begin{gather}
P[X]
=
P[X_0] + \sum_{d=1}^m P[X_d \setminus X_{d-1}].
\end{gather}
We will use this strategy to compute $p_r$ and $P_{(\bn^i)}$, where these are abbreviations for the following virtual Poincar\'{e} polynomials:
\begin{gather}
p_r
\coloneqq
P\Bigl[\ol M_r^\bC\Bigr],
\quad
P_{(\bn^i)}
\coloneqq
P\Biggl[\prod_i^{\ol M_r^\bC} \ol{2M}_{\bn^i}^\bC\Biggr].
\end{gather}

We will stratify $\ol M_r^\bC$ in terms of the root screen.
More specifically, on the root screen, there is an output marked point and $\geq 2$ nodes and input marked points.
The latter points give a partition of $\{1,\ldots,r\}$, and we stratify $\ol M_r^\bC$ by these partitions.
Similarly, we stratify $\ol{2M}_{(\bn^i)}^\bC$ in terms of how the lines and points have collided, from the root screen's perspective.
Using Schiessl's computation \cite{schiessl} of the virtual Poincar\'{e} polynomial of the ordered configuration space $F_\ell(\bC\setminus k)$ of $\ell$ points in $\bC \setminus k$, we can recursively compute the virtual Poincar\'{e} polynomial of each stratum, which allows us to algorithmically compute $p_r$ and $P_{(\bn^i)}$.
Of course, this is more than necessary, since we care primarily about $P[\ol{2M}_\bn^\bC]$.
The reason that we compute the virtual Poicar\'{e} polynomial of arbitrary fiber products of $\ol{2M}_\bn^\bC$ is that each stratum of $\ol{2M}_\bn^\bC$ decomposes as the product of copies of $F_\ell(\bC\setminus k)$ and of fiber products of $\ol{2M}_\bn^\bC$, which is key to our recursion.
The analogous statement is not true for a single $\ol{2M}_\bn^\bC$: each stratum of this space again decomposes as the product of copies of $F_\ell(\bC\setminus k)$ and of fiber products of $\ol{2M}_\bn^\bC$, so we would not be able to recurse.

\begin{table}
\centering
\begin{tabular}{|c|c||c|c|}
\hline
\multicolumn{2}{|c||}{$d=2$} & \multicolumn{2}{|c|}{$d=3$} \\
\hline
$\bn$ & $P(x)$ & $\bn$ & $P(x)$ \\
\hline
$(1,2)$ & $x^4 + 4x^2 + 1$ & $(0,0,0,2)$ & $x^6 + 12x^4 + 12x^2 + 1$ \\
\hline
$(0,1,1)$ & $x^4 + 3x^2 + 1$ & $(1,3)$ & $x^6 + 15x^4 + 12x^2 + 1$ \\
\hline
$(0,0,2)$ & $x^4 + 4x^2 + 1$ & $(0,1,2)$ & $x^6 + 10x^4 + 10x^2 + 1$ \\
\hline
$(0,0,0,1)$ & $x^4 + 5x^2 + 1$ & $(2,2)$ & $x^6 + 14x^4 + 11x^2 + 1$ \\
\hline
$(4)$ & $x^4 + 5x^2 + 1$ & $(5)$ & $x^6 + 16x^4 + 16x^2 + 1$ \\
\hline
$(0,3)$ & $x^4 + 5x^2 + 1$ & $(0,0,1,1)$ & $x^6 + 9x^4 + 9x^2 + 1$ \\
\hline
& & $(0,4)$ & $x^6 + 19x^4 + 16x^2 + 1$ \\
\hline
& & $(0,0,0,0,1)$ & $x^6 + 16x^4 + 16x^2 + 1$ \\
\hline
& & $(1,1,1)$ & $x^6 + 8x^4 + 8x^2 + 1$ \\
\hline
& & $(0,0,3)$ & $x^6 + 14x^4 + 14x^2 + 1$ \\
\hline
\end{tabular}
\medskip
\caption{\label{tab:VPP}}
\end{table}

For any $r \geq 1$ and $\bn^1,\ldots,\bn^a \in \bZ_{\geq0}^r\setminus\{\bzero\}$, define $\sP_r$ and $2\sP_{(\bn^i)}$ like so:
\begin{gather}
\sP_r
\coloneqq
\text{partitions of } \{1,\ldots,r\},
\\
2\sP_{(\bn^i)}
\coloneqq
\Bigl\{\Bigl(P,(2P_{p,i})_{{p \in P}\atop{1\leq i\leq a}}\Bigr)
\:\Big|\:
P \in \sP_r,
\:
2P_{p,i} \text{ a partition of } \prod_{j \in p} \bigl\{(i,j,1),\ldots,(i,j,n_j^i)\bigr\}
\Bigr\}
\nonumber
\end{gather}
(Our partitions are unordered, and by convention a partition of $S$ cannot contain $\emptyset$ unless $S = \emptyset$.)
For any $\bigl(P,(2P_{p,i})\bigr) \in 2\sP_{(\bn^i)}$, we denote by $\ul\#2P_{p,i}$ the vector which, for $j \in p$, has $(\ul\#2P_{p,i})_j$ equal to the number of elements of $2P_{p,i}$ of the form $(i,j,k)$.
Next, we define the stable elements of $\sP_r$ and $2\sP_{(\bn^i)}$:
\begin{gather}
\sP_r^\stab
\coloneqq
\bigl\{P \in \sP^r \:|\: \#P \geq 2\bigr\},
\\
2\sP_{(\bn^i)}^\stab
\coloneqq
\bigl\{\bigl(P,(2P_{p,i})\bigr) \in \sP_r \:|\: \#P \geq 2 \text{ or } P = \{\{1,\ldots,r\}\}, \#\bigl(P_{\{1,\ldots,r\},i}\bigr) \geq 2 \:\forall\: i\bigr\}.
\nonumber
\end{gather}
These sets of partitions allow us to stratify $\ol M_r^\bC$ and $\ol{2M}_\bn^\bC$, according to the type of the root component:
\begin{gather}
\ol M_r^\bC
=
\bigsqcup_{P \in \sP_r} \ol M_{r,P}^\bC,
\qquad
\ol{2M}_\bn^\bC
=
\bigsqcup_{(P,(2P_{p,i})) \in 2\sP_{(\bn^i)}} \ol{2M}_{\bn,(P,(2P_{p,i}))}^\bC.
\end{gather}

We can then make the following identifications for any $P \in \sP_r$ and $\bigl(P,(2P_{p,i})\bigr) \in 2\sP_{(\bn^i)}$:
\begin{gather}
\label{eq:strata_recurse_decomp}
\ol M_{r,P}^\bC
\simeq
\bigl(\bC^{\#P}\setminus\Delta\bigr)/_{\bC \rtimes \bC^*}
\times
\prod_{p \in P}
\ol M_{\#p}^\bC,
\\
\prod_i^{\ol M_r^\bC} \ol{2M}_{\bn^i}^\bC
\simeq
\bigl(\bC^{\#P}\setminus\Delta\bigr)/_{\bC\rtimes\bC^*}
\times
\prod_{1\leq i\leq a} \biggl(\prod_{p \in P} \bC^{\#2P_{p,i}}\setminus\Delta\biggr)/_\bC
\times
\prod_{p \in P}
\prod_{{1\leq i\leq a}
\atop
{2p \in 2P_{p,i}}}^{\ol M_{\#P}^\bC} \ol{2M}_{\ul\#2P_{p,i}}^\bC.
\nonumber
\end{gather}
To use \eqref{eq:strata_recurse_decomp} to recursively compute $p_r$ and $P_{(\bn^i)}$, we need the following result:

\medskip

\noindent
{\bf Theorem 4.1, \cite{schiessl}.}
{\it The virtual Poincar\'{e} polynomial of the unordered configuration space of $\ell$ points in $\bC^k$ is given by the following formula:
\begin{align}
\label{eq:VPP_Fn}
P\bigl[(\bC\setminus k)^\ell\setminus\Delta\bigr](x)
=
(x^2 - k)(x^2 - k - 1)\cdots(x^2 - k - n + 1).
\end{align}}

\medskip

\noindent
The following identities now follow from \eqref{eq:VPP_decomposition} and \eqref{eq:VPP_Fn}:
\begin{align}
P\Bigl[\ol M_{r,P}^\bC\Bigr]
&=
(x^2-2)(x^2-3)\cdots(x^2-\#P+1)
\prod_{2p \in 2P_{p,i}}
P\Bigl[\ol M_{\#p}^\bC\Bigr],
\\
P\Biggl[\prod_i^{\ol M_r^\bC} \ol{2M}_{\bn^i}^\bC\Biggr]
&=
(x^2-2)(x^2-3)\cdots(x^2-\#P+1)\times
\nonumber
\\
&\qquad
\times
\prod_{1\leq i\leq a}\biggl(\frac1{x^2}\prod_{p \in P} x^2(x^2-1)\cdots(x^2-\#2P_{p,i}+1)\biggr)
\prod_{p \in P} P\Biggl[\prod_{{1\leq i\leq a}
\atop
{2p \in 2P_{p,i}}}^{\ol M_{\#P}^\bC} \ol{2M}_{\ul\#2P_{p,i}}^\bC\Biggr].
\nonumber
\end{align}
These identities allow us to recursively compute $p_r$ and $P_{(\bn^i)}$ in \cite{bo:program}.

\end{document}